\documentclass[12pt]{amsart}
\usepackage{amsmath,amssymb,amsthm,amscd,amsfonts}
\usepackage{verbatim}
\usepackage{comment}
\usepackage{multirow}
\usepackage{mathtools}
\usepackage{enumitem}
\usepackage{pgf,tikz}
\usepackage{tikz-cd}
\usetikzlibrary{arrows}
\usetikzlibrary{cd} 
\usepackage[normalem]{ulem}
\usepackage{graphicx}
\usepackage{wasysym}
\usepackage{csquotes}
\usepackage{bbm}
\usepackage{bm}
\usepackage{color}
\usepackage{stmaryrd}
\usepackage[margin=1in]{geometry} 
\usepackage{mathbbol}
\usepackage{pdfpages}
\usepackage{hyperref}

\newcommand{\Z}{\mathbb{Z}}
\newcommand{\PP}{\mathbb{P}}
\newcommand{\E}{\mathcal{E}}
\newcommand{\LL}{\mathcal{L}}
\newcommand{\C}{\mathcal{C}}
\newcommand{\En}{\mathcal{E}_{norm}}
\newcommand{\OO}{\mathcal{O}}
\newcommand{\kk}{\mathbb{k}} 
\DeclareMathOperator{\HH}{H}

\DeclareMathOperator{\depth}{depth}
\DeclareMathOperator{\Ext}{Ext}
\DeclareMathOperator{\cod}{codim}
\DeclareMathOperator{\cok}{coker}
\DeclareMathOperator{\exc}{expcodim}

\DeclareMathOperator{\rk}{rank}
\DeclareMathOperator{\ch}{char}

\DeclarePairedDelimiter\floor{\lfloor}{\rfloor}

\newcommand{\verteq}{\rotatebox{270}{$\,\footnotesize{\cong} \ $}}
\newcommand{\equalto}[2]{\underset{\scriptstyle\overset{\mkern4mu\verteq}{#2}}{#1}}

\theoremstyle{plain}
\newtheorem*{prop*}{Proposition}
\newtheorem*{teo*}{Theorem}
\newtheorem*{cor*}{Corollary}

\newtheorem{teo}{Theorem}[section]
\newtheorem{cor}[teo]{Corollary}
\newtheorem{con}[teo]{Conjecture}
\newtheorem{prop}[teo]{Proposition}
\newtheorem{lemma}[teo]{Lemma}
\theoremstyle{definition}
\newtheorem{df}[teo]{Definition}
\newtheorem{ex}[teo]{Example}
\theoremstyle{remark}
\newtheorem{rmk}[teo]{Remark}

\title{The non-Lefschetz locus of conics}
\author{Emanuela Marangone}
\email{emanuela.marangone@umanitoba.ca}
\address{Department of Mathematics \\
	University of Manitoba \\
	Winnipeg, MB R3T2N2, Canada}
\subjclass[2020]{13E10, 13F20, 13D02, 13C40, 14M05 (primary); \\ 13H10, 14H60, 14M12, 14M10, 13A02, 14F06 (secondary)}
\keywords{Weak Lefschetz property, Strong Lefschetz property, Rank 2 vector bundles, Jumping conics, Jumping lines}
\thanks{First and foremost, I would like to express my sincere gratitude to my Ph.D. advisor Juan Migliore, for his guidance during this project. A special thanks to Eric Riedl, Matthew Scalamandre, Susan Cooper, and Matthew Weaver for helpful discussions regarding various aspects of this project.\\
The results in this article form part of my Ph.D. thesis at the University of Notre Dame.
I would also like to acknowledge that this research was supported in part by the Pacific Institute for the Mathematical
Sciences.}

\begin{document}

	\begin{abstract}
A graded Artinian algebra $A$  has the \emph{Weak Lefschetz Property} if there exists a linear form $\ell$ such that the multiplication map $\times\ell:[A]_i\to [A]_{i+1}$ has maximum rank in every degree. The linear forms satisfying this property form a Zariski-open set; its complement is called the \emph{non-Lefschetz locus} of $A$.

In this paper, we investigate analogous questions for degree-two forms rather than lines. We prove that any complete intersection $A=\kk[x_1,x_2,x_3]/(f_1,f_2,f_3)$, with $\ch \kk=0$, has the Strong Lefschetz Property at range $2$, i.e. there exists a linear form $\ell\in [R]_1$, such that the multiplication map $\times \ell^2: [A]_i\to [A]_{i+2}$ has maximum rank in each degree.

Then, we focus on the forms of degree 2 such that $ \times C: [A]_i\to [A]_{i+2}$ fails to have maximum rank in some degree $i$. The main result shows that the \emph{non-Lefschetz locus of conics} for a general complete intersection $A=\kk[x_1,x_2,x_3]/(f_1,f_2,f_3)$ has the expected codimension as a subscheme of $\PP^5$. The hypothesis of generality is necessary. We include examples of monomial complete intersections for which the non-Lefschetz locus of conics has different codimension. 

To extend a similar result to the first cohomology modules of rank $2$ vector bundles over $\PP^2$, we explore the connection between non-Lefschetz conics and \emph{jumping conics}. The non-Lefschetz locus of conics is a subset of the jumping conics. Unlike the case of lines, this can be a proper subset when $\E$ is semistable with first Chern class even. 
	\end{abstract}

    	\maketitle

\tableofcontents

\section{Introduction}

A graded Artinian algebra $A$  has the \emph{Weak Lefschetz Property (WLP)} if the multiplication map by a linear form $\ell$ has maximum rank in every degree. If the multiplication map by any power of a general linear form $\ell$ has maximum rank as well, we say that $A$ has the \emph{Strong Lefschetz Property (SLP)}.
This last property also implies the \emph{Maximal Rank Property (MRP)} i.e, for any $d$, the multiplication by a general form of degree $d$ always has maximum rank.
In this paper, we study the forms of degree 2 such that $ \times C: [A]_i\to [A]_{i+2}$ fails to have maximum rank in some degree $i$. 

The most famous result in the study of the Lefschetz Property states that every Artinian monomial complete intersection over a field $\kk$ of characteristic zero has the SLP \cite{28}, \cite{26} \cite{29}. As a consequence of this theorem, a general complete intersection $\kk[x_1,\dots,x_n]/(f_1,\dots, f_n)$ with fixed generator degrees has the SLP. However, it is an open question to determine whether every complete intersection has the SLP or even the WLP.  The most important result in this direction proves the Weak Lefschetz Property for $n=3$. To prove this theorem, Harima--Migliore--Nagel--Watanabe \cite{13} applied the Grauert--M\"{u}lich Theorem to the syzygy bundle, in this case, a locally free sheaf of rank $2$ over $\PP^2$. 

The study of the \emph{non-Lefschetz locus}, defined as the set of linear forms that fail to have maximum rank in some degree,  started with the work of Boij--Migliore--Miró-Roig--Nagel \cite{main}. In  \cite{main}, the authors conjecture that the non-Lefschetz locus of a general complete intersection has the expected codimension. This conjecture has been proven for general complete intersections if $n=3$ \cite{main}. 
Failla--Flores--Peterson \cite{FFP21} generalize the study of the Lefschetz Property to the first cohomology module $\HH_*^1(\PP^2,\E)$ of any vector bundle $\E$ of rank $2$ over $\PP^2$ and prove that these finite-length modules have the WLP.  The non-Lefschetz locus of $\HH_*^1(\PP^2,\E)$ is exactly the set of jumping lines, and it has the expected codimension under the assumption that $\E$ is general \cite{vb}.

We investigate analogous questions for higher-degree forms rather than lines. This problem is connected with the SLP at range $2$ and with the Maximal Rank Property (MRP). In analogy with \cite{main}, we endow  
 $$\mathcal{C}_{A,i}=\{C\in[R]_2 \ | \times C: [A]_i\to [A]_{i+2} \text{ does not have max rank} \}$$ with a scheme structure given by the ideal $I(\mathcal{C}_{A,i})$ of maximal minors of a suitable $h_{i+2} \times h_{i}$  matrix of linear forms. $\mathcal{C}_A$ is the subscheme of $\PP^{\binom{n+1}{2}-1}$ defined by the ideal $I(\mathcal{C}_A)=\bigcap I(\mathcal{C}_{A,i})$.
In this paper, we will mainly focus on the case $n=3$. In this case, we refer to $\C_A$ as \emph{the non-Lefschetz locus of conics of $A$}.
The main result provides a version of \cite[Theorem 5.3]{main} for conics.
\begin{teo*}[\ref{main CI}]
	The non-Lefschetz locus of conics for a general complete intersection of height $3$ has the expected codimension and degree as a subscheme of $\PP^5$.
\end{teo*}

This paper is organized as follows.
In \textbf{\S 2}, following what has been done for lines in \cite{main}, we define the \emph{non-Lefschetz locus for forms of degree $2$} for any Artinian algebra $A$ over $R=\kk[x_1,\dots, x_n]$ as a subscheme of $\PP^{\binom{n+1}{2}-1}$. We will prove that
\begin{prop*}[\ref{'mid}]
Let $A$ be a Gorenstein Artinian algebra of socle degree $e$. If $A$ has the WLP, then the non-Lefschetz locus of forms of degree $2$ is defined by the ideal in the middle degree. 
\end{prop*}

In \textbf{\S 3}, we specialize to the case in  $3$ variables. In this setting, we refer to $\mathcal{C}_A$ with its structure as a subscheme of $\PP^5$ as the \emph{non-Lefschetz locus of conics} of $A$.
Our goal is to study the non-Lefschetz locus of conics for the first cohomology modules $\HH_*^1(\PP^2,\E)$ of rank $2$ vector bundles $\E$, as a generalization of the case of complete intersections  $R/(f_1,f_2,f_3)$. Therefore in this section, we provide some background on vector bundles over $\PP^2$.

In \textbf{\S 4}, we give the following result.
 \begin{cor*}[\ref{SLP2}]
	Any complete intersection $A=\kk[x_1,x_2,x_3]/(f_1,f_2,f_3)$ has the Strong Lefschetz Property at range $2$.
  \end{cor*} 
   Moreover, the lines $\ell$ for which the multiplication map by $\ell^2$ does not have maximal rank form a hypersurface in the space of linear forms. 
 The same is true as well for any first cohomology module of a rank $2$ vector bundle $\HH_*^1(\PP^2,\E)$ over $\PP ^2$ (Proposition \ref{LC}).
As a consequence the non-Lefschetz locus of conics of $M=\HH_*^1(\PP^2,\E)$ has positive codimension in $\PP^5$.

In \textbf{\S 5}, we investigate the connection between non-Lefschetz conics and the jumping conics. The notion of \emph{jumping conics} was initially introduced by \cite{Vitter} for semistable vector bundles of rank $2$ over $\PP^2$.  We first extend the definition to include the case when $\E$ is unstable. It follows that whenever $\E$ is unstable or  $c_1(\E)$ is odd, the jumping conics are exactly the non-Lefschetz conics. In such cases, the non-Lefschetz locus of conics $\C_M$ is a hypersurface in $\PP^5$. This does not hold when $\E$ is semistable with $c_1(\E)$ even: while every non-Lefschetz conic is indeed a jumping conic, the reverse is not necessarily true.
\begin{cor*}[\ref{jum3}]
	Let $\E$ be a semistable, normalized vector bundle with $c_1(\E)=0$.\\
 A smooth conic $C$ is a non-Lefschetz conic if and only if it is a jumping conic such that $\E_{| C}\cong \OO_{\PP^1}(a)\oplus\OO_{\PP^1}(-a)$ with $a>1$.
\end{cor*}
In \textbf{\S 5.1}, we compute the expected codimension of $\C_M$. For a semistable $\E$ with even first Chern class, the non-Lefschetz locus of conics is expected to have codimension $2$ if $\E$ is semistable but not stable, and $3$ if $\E$ is stable. We conjecture that such dimension is achieved for $\E$ general. 

In \textbf{\S 6}, we resolve the question about the codimension of the non-Lefschetz locus of conics for a general complete intersection, proving the conjecture for this particular case.
\begin{teo*}[\ref{main CI}]
	Let $A=R/(f_1,f_2,f_3)$ be a general complete intersection of type $(d_1,d_2,d_3)$, and socle degree $e$. The non-Lefschetz locus of conics has the expected codimension in $\PP^5$:
	$$\cod \mathcal{C}_{A}=\begin{cases} 1& \text{ if } e \text{ is even or } d_3\geq d_1+d_2+2;\\
		2& \text{ if } d_3=d_1+d_2\\
		3& \text{ if } e \text{ odd and }d_3\leq d_1+d_2-2.\\
	\end{cases}$$
\end{teo*}
The most interesting case is when the socle degree $e$ is odd and $d_3\leq d_1+d_2$. For this case, in \textbf{\S 6.1} we construct an explicit Gorenstein algebra $R/J$  with the desired Hilbert function for which the non-Lefschetz locus of conics has expected codimension and then invoke semicontinuity.
 In \textbf{\S 6.2}, for the case $d_3< d_1+d_2-2$,  we also prove that the set of conics in $\C_A$ that do not vanish at any of the points of the zero-dimensional scheme defined by the ideal $(f_1,f_2)$, has codimension  $3$ in $\PP^5$.

The hypothesis of generality is necessary whenever the socle degree $e$ is odd and $d_3\leq d_1+d_2$. In fact, in \textbf{\S 7}  we construct examples of monomial complete intersection $A=R/(x_1^{d_1},x_2^{d_2},x_3^{d_3})$ 
with $d_3\leq d_1+d_2-2$, where $\C_A$ has codimension 1, 2, or 3.
Moreover, we show that for every monomial complete intersection with $d_3=d_1+d_2$, $\C_A$ is always a hypersurface even if, in this case, the expected codimension is $2$.

In \textbf{\S 8}, we study the non-Lefschetz locus of conics for general height $3$ Gorenstein algebras with fixed Hilbert function. As in the case of lines studied in \cite{main}, to have the expected codimension we need an extra condition on the \emph{g-vector} $(1,2,g_2,\dots,g_{\lfloor \frac{e}{2}\rfloor})$, defined as the positive part of the first difference of the Hilbert function.
\begin{prop*}[\ref{genGo}]
	Let $(1,3,h_2,\dots,h_{e})$ be an SI-sequence such that its g-vector is of decreasing type. 
	Then for a general Gorenstein algebra with Hilbert function $(1,3,h_2,\dots,h_{e})$ the non-Lefschetz locus of conics has expected codimension in $\PP^5$.
 \end{prop*}
 Without the condition on the first difference, the expected codimension might not be achieved. Given an SI-sequence $(1,3,\dots,h_{e})$ with g-vector not of decreasing type, the non-Lefschetz locus of conics of a general Gorenstein algebra with Hilbert function $(1,3,\dots,h_{e})$, always has codimension 1.

\section{The non-Lefschetz locus of forms of degree $2$}

In this section, we want to define the non-Lefschetz locus for forms of degree $2$ for any Artinian algebra, following what has been done for linear forms in \cite{main}.
We first recall the definitions of Lefschetz Properties and non-Lefschetz locus.  

Let $\kk$ be an algebraically closed field of characteristic $0$, and let $R=\kk[x_1,\dots,x_n]$  be the standard graded polynomial ring in $n$ variables.
Let $A=R/I=\bigoplus_{i=0}^{e} [A]_i$ be a graded Artinian algebra of socle degree $e$. We denote with $(h_0,\dots, h_e)$ the Hilbert function of $A$, i.e. $h_i=\dim ([A]_i)$ for each $0\leq i\leq e$. 
\begin{df}
$A$ has the \emph{Weak Lefschetz Property (WLP)} if there exists a linear element $\ell$ in $R$ such that the multiplication map $\times\ell: [A]_i\to [A]_{i+1}$ has maximal rank for every  degree $i$, i.e. it is always either surjective or injective. Such an element is called a \emph{(Weak) Lefschetz element}.
\end{df}
The set of Lefschetz elements forms a Zariski-open set, and its complement is called the \emph{non-Lefschetz locus} of $A$, which is denoted by $\LL_A$. 
\begin{df}
	We say that $A$ has the \emph{Strong Lefschetz Property (SLP)} if there is a linear element $\ell$ such that the multiplication map for any power of $\ell$, $\times\ell^d:[A]_i\to [A]_{i+d}$, has maximal rank in each degree $i$.
\end{df}

The SLP also implies the Maximal Rank Property:
\begin{df}
$A$ has the \emph{Maximal Rank Property (MRP)} if for each $d$ the multiplication map for any general form of degree $d$, $f\in[R]_d$, has maximal rank in each degree.
\end{df}
If we fix $d$ we also have
\begin{df}
	We say that $A$ has the \emph{Strong Lefschetz Property (SLP) at range $d$} if there is a linear element $\ell$ such that the multiplication map $\times\ell^d:[A]_i\to [A]_{i+d}$, has maximal rank in each degree $i$.
\end{df}

In this paper, we want to focus only on the case $d=2$, and study the forms $C\in [R]_2$ such that there exists a degree $i$ for which the multiplication map $\times C: [A]_i\to [A]_{i+2}$ is neither surjective nor injective.

	\begin{df}\label{C}
		For the Artinian algebra $A$ of socle degree $e$, we define the \emph{non-Lefschetz locus of forms of degree 2} as
		$$\C_A=\bigcup_{i=0}^{e-2} \C_{A,i}$$ where 
            $$\C_{A,i}=\{C\in [R]_2: \ \times C: [A]_i\to [A]_{i+2} \text{ does not have maximum rank } \}.$$
	\end{df}

We want to study the non-Lefschetz locus of forms of degree $2$ as a subscheme of $\PP^{N-1}$, for $N=\binom{n+1}{2}$, in a similar way to what has been done for the non-Lefschetz locus $\LL_A$ in \cite{main}. 

For each $i$  we have a map 
\begin{align*}
	\phi :[R]_2 &\to \hom_k([A]_i,[A]_{i+2})\\
	C&\mapsto \times C: [A]_i \to [A]_{i+2}.
\end{align*}
Given a choice of basis for $[A]_i$ and $[A]_{i+2}$ as $\kk$-vector spaces, $\phi(C)\in\hom_k([A]_i,[A]_{i+2})$ is represented by a $h_{i+2}\times h_{i}$ matrix whose entries linearly depend on the coefficients of $C$. Then $C\in\C_{A,i}$ if and only if the $\rk \phi(C) < \min\{h_i,h_{i+2}\}$, i.e. if one of the maximal minors of $\phi(C)$ is zero, and this does not depend on the choice of the basis.
 
Following this idea, we define $S=\kk[a_1,\dots,a_N]$ to be the coordinate ring associated with $\PP^{N-1}$. We can think of the variables $a_1,\dots,a_N$ as the coefficients of a conic $$C=a_1x_1^2+a_2x_1x_2+\dots+a_Nx_n^2.$$

The multiplication map $\times C$ on $S\otimes_k A$ gives us  a map 
$$\times C : S\otimes_k [A]_i \to S\otimes_k [A]_{i+2}$$
that, with chosen bases for $[A]_i$ and $[A]_{i+2}$, can be represented by a $h_{i+2}\times h_{i}$ matrix of linear forms in $S$. We call this matrix  $B_i$.
  
\begin{df}
	$\C_{A,i}\subseteq \PP^{N-1}$ is scheme-theoretically defined by the ideal $I(\C_{A,i})$ of maximal minors of the matrix $B_i$. The \emph{non-Lefschetz locus of forms of degree 2}, $\C_A$, is defined as a subscheme of $\PP^{N-1}$ by the homogeneous ideal $I(\C_A)=\bigcap_{0\leq i\leq e-2}I(\C_{A,i})$.
\end{df}
\begin{rmk}
    This definition can be applied to forms of higher degrees as well; in this paper, we analyze just the case of degree $2$.
\end{rmk}
The expected codimension  of $\C_{A,i}$ is $|h_{i+2}-h_i|+1$ and if this codimension is achieved then $\deg \C_{A,i}=\binom{h_{i+2}}{h_i-1}$  by \cite{MIGLIORE}. 

The non-Lefschetz locus $\C_A$ of forms of degree-two forms is defined a priori as a union of determinantal schemes. However, to get a clearer understanding (and in particular to facilitate the computation of the expected codimension) we need to understand which condition we need in order to get a containment between the ideals $I(\C_{A,i})$ and determine that $\C_A$ is ``concentrated'' in one degree.

Similarly to what is done for lines in \cite{main}, we have the following proposition.
\begin{prop} \label{matrix}
		If $h_i\leq h_{i+2}\leq h_{i+4}$ and there is no socle in degree $i$ and $i+1$, then $I(\C_{A,i+2})\subseteq I(\C_{A,i})$.
\end{prop} 
\begin{proof}
The proof is divided into three parts. First, we will show that we can reduce to the case $h_{i+2}=h_{i+4}$. In the second step, we take care of the case when one of the ideals is $0$. Finally, in the third part, we assume $h_{i+2}=h_{i+4}$ and $I(\C_{A,i+2})\neq0\neq I(\C_{A,i})$ and prove the desired inclusion.

\emph{Step 1.} Here we reduce to the case when $h_{i+2}=h_{i+4}$.

Fix bases for $[A]_{i+2}$ and $[A]_{i+4}$. Recall that $I(\C_{A,i+2})$ is the ideal of maximal minors of the $h_{i+4}\times h_{i+2}$ matrix $B_{i+2}$, given by the map
$$S\otimes_k [A]_{i+2} \stackrel{B_{i+2}}\to S\otimes_k [A]_{i+4}.$$
Denote with $\{v_1,\dots, v_{h_{i+4}}\}$ the chosen basis for $[A]_{i+4}$. For each  $J\subseteq \{1,\dots, h_{i+4}\}$ consisting of $h_{i+2}$ elements, we define  $A^J$ to be the algebra obtained by quotienting $A$ by the ideal generated by the forms of degree $i+4$ indexed by elements not in $J$:
$$A^J=\frac{A}{(v_j \ | \ j\notin J)}.$$

Since we quotient by an ideal generated in degree $i+4$, $[A^J]_{j}=[A]_j$ for every $j<i+4$. In particular, $[A^J]_i=[A]_{i}$ and $[A^J]_{i+1}=[A]_{i+1}$, so $I(\C_{A^J,i})=I(\C_{A,i})$. 
Consider the diagram 
 \[ \begin{tikzcd}
S\otimes_k [A]_{i+2} \arrow{r}{B_{i+2}}\arrow[equal]{d}
& S\otimes_k [A]_{i+4} \arrow[two heads]{d}\\
S\otimes_k [A^J]_{i+2} \arrow{r}{B^J}&S\otimes_k [A^J]_{i+4}.
\end{tikzcd}
\]
Under these identifications, the lower map is represented by the submatrix $B^J$, obtained from $B_{i+2}$ by considering just the rows indexed by elements of $J$. In fact, $B^J$ is a maximal square submatrix of $B_{i+2}$, and all the maximum submatrices are of this form, for some $J\subseteq\{1,\dots, h_{i+4}\}$ with  $|J|=h_{i+2}$.
Since $I(\C_{A^J,i+2})$ is defined to be the ideal of maximal minors of $B^J$, we have $(\det B^J)=I(\C_{A^J,i+2})$.

In order to prove that $I(\C_{A,i+2})\subseteq I(\C_{A,i})$, we need to show that each maximal minor of the  $B_{i+2}$ is in the ideal $I(\C_{A,i})$, equivalently that for each $J$, $(\det B^J)\subseteq I(\C_{A,i})$. By the equalities proved above this is the same as to prove that  $I(\C_{A^J,i+2})\subseteq I(\C_{A^J,i})$, for each  $J\subseteq \{1,\dots, h_{i+4}\}$ subset of cardinality $h_{i+2}$.

If we prove the proposition for each  Artinian algebra with  $h_i\leq h_{i+2}=h_{i+4}$ and no socle in degree $i$ and $i+1$, then the statement will follow.


\emph{Step 2.} Consider the case when one of the ideals is zero.
If $I(\C_{A,i+2})=0$, then there is nothing to prove. When $I(\C_{A,i})=0$, it is enough to show that $\C_{A,i+2}=\PP^{N-1}$. 
Assume by contradiction that $\exists C\in [R]_2$ such that the multiplication map $\times C:[A]_{i+2}\to [A]_{i+4}$ is injective. Since we assume $I(\C_{A,i})=0$, $\C_{A,i}=\PP^{N-1}$and, in particular, there exists a $f\in [A]_i$ non-zero with $Cf=0$. $A$ does not have socle in degree $i$ or $i+1$, so $\exists \ell, \ell' \in [R]_1$ such that $f\ell\ell'\neq 0$ in $[A]_{i+2}$. But $C f\ell\ell'=0 $ contradicting injectivity of $\times C:[A]_{i+2}\to [A]_{i+4}$.

\emph{Step 3.} Let us assume $h_{i+4}=h_{i+2}$ and $\C_{A,i+2}\neq \PP^{N-1}\neq  \C_{A,i}$.
Then we can choose a conic $C\in[R]_2$, such that $\times C:[A]_{i+2}\to [A]_{i+4}$ is bijective, and so $\times C:[A]_{i}\to [A]_{i+2}$ injective, by the argument in Step 2. We have the diagram
\[ \begin{tikzcd}
S\otimes_k [A]_{i} \arrow{r}{B_{i}} \arrow[hook, swap]{d}{\times C \ } & S\otimes_k [A]_{i+2} \arrow{d}{\rotatebox{270}{$\sim \ $}}{\ \ \times C}\\%
S\otimes_k [A]_{i+2} \arrow{r}{B_{i+2}}& S\otimes_k [A]_{i+4}.
\end{tikzcd}
\]
Since $C$ is fixed the vertical maps are the identity on $S$. Hence, chosen appropriate basis for $[A]_{i+2}$ and $[A]_{i+4}$, $B_i$ is a submatrix of $B_{i+2}$.
We assume $h_{i+2}=h_{i+4}$, so $B_{i+2}$ is a square matrix, and $B_i$ is a submatrix with the same number of rows $h_{i+2}$ as $B_{i+2}$.  Then,  $\det B_{i+2}$ can be written as a linear combination of the maximal minors of $B_i$, and
$$I(C_{A,i+2})=(\det B_{i+2})\subseteq(\text{max minors of } B_i)=I(C_{A,i})$$
as desired.

\end{proof}

\begin{rmk}\label{dual}
Proposition \ref{matrix} holds if we consider finite length modules $M$ instead of Artinian algebras. Moreover, we also have a dual statement: if $h_i\geq h_{i+2}\geq h_{i+4}$ and there are no new generators in degree $i+3$ and $i+4$, i.e. $(x_1,\dots,x_n)[M]_{i+2}=[M]_{i+3}$ and $(x_1,\dots,x_n)[M]_{i+3}=[M]_{i+4}$, then $I(\C_{M,i+2})\supseteq I(\C_{M,i})$.
\end{rmk}
 
\begin{teo}\label{'mid}
	Let $A$ be a Gorenstein Artinian algebra of socle decree $e$. If $A$ has the Weak Lefschetz Property, then the non-Lefschetz locus of forms of degree $2$ is defined as a subscheme of $\PP^{N-1}$ by the ideal in the middle degree $$I\big(\C_A\big)= I\big(\C_{A,\lfloor\frac{e}{2}\rfloor-1}\big).$$
\end{teo}
\begin{proof}
	The WLP guarantees the unimodality of the Hilbert function. Then applying Proposition \ref{matrix} we
   have that 
  \begin{align*}
    &e \text{ even: }&\vspace{5pt} & \cdots\leq h_{\frac{e}{2}-4}\leq h_{\frac{e}{2}-2}\leq h_{\frac{e}{2}} & \Rightarrow & \vspace{5pt} & \cdots\supseteq I(\C_{A,\frac{e}{2}-4})\supseteq I(C_{A,\frac{e}{2}-2}) \\
     &               && \cdots\leq h_{\frac{e}{2}-3}\leq h_{\frac{e}{2}-1}= h_{\frac{e}{2}+1}& \Rightarrow & \vspace{5pt} &\cdots\supseteq I(\C_{A,\frac{e}{2}-3})\supseteq I(\C_{A,\frac{e}{2}-1}) \\
    &e \text{ odd: }& \vspace{5pt} & \cdots\leq h_{\frac{e-1}{2}-4}\leq h_{\frac{e-1}{2}-2}\leq h_{\frac{e-1}{2}} & \Rightarrow & \vspace{5pt} & \cdots\supseteq I(\C_{A,\frac{e-1}{2}-4})\supseteq I(\C_{A,\frac{e-1}{2}-2})& \\
     &               && \cdots\leq h_{\frac{e-1}{2}-3}\leq h_{\frac{e-1}{2}-1}\leq h_{\frac{e+1}{2}} & \Rightarrow &\vspace{5pt} &  \cdots\supseteq I(\C_{A,\frac{e-1}{2}-3})\supseteq I(\C_{A,\frac{e-1}{2}-1}) \end{align*}
where  $e$ is the socle degree of $A$.

 Since $A$ is Gorenstein, duality shows that, for each $i\geq 0$, $I\big(\C_{A,\lfloor\frac{e}{2}\rfloor-i-2}\big)=I\big(\C_{A,\lfloor\frac{e+1}{2}\rfloor+i}\big)$
  and so we have
  $$I\big(\C_A\big)=\bigcap_{i=0}^{e-2}I\big(\C_{A,i}\big)=I\big(\C_{A,\lfloor\frac{e}{2}\rfloor-1}\big)\cap I\big(\C_{A,\lfloor\frac{e}{2}\rfloor-2}\big).$$
  
 To prove 
 $I\big(\C_A\big)= I\big(\C_{A,\lfloor\frac{e}{2}\rfloor-1}\big)$
 we will show
 $$I\big(\C_{\lfloor\frac{e}{2}\rfloor-1}\big)\subseteq I\big(\C_{\lfloor\frac{e}{2}\rfloor-2}\big)$$
 using the assumption that $A$ has the WLP. We will divide the proof depending on the parity of the socle degree $e$.
 
 \emph{Case $e$ odd.} Fix a Weak Lefschetz element $\ell \in [R]_1$.
 Then $\times\ell:[A]_{\frac{e-1}{2}-2}\to [A]_{\frac{e-1}{2}-1}$ is injective and $\times\ell:[A]_{\frac{e-1}{2}}\to [A]_{\frac{e+1}{2}}$ is bijective ($h_{\frac{e-1}{2}}=h_{\frac{e+1}{2}}$). We have the diagram
 \[ \begin{tikzcd}
S\otimes_k [A]_{\frac{e-1}{2}-2} \arrow{r}{B_{\frac{e-1}{2}-2}} \arrow[hook, swap]{d}{\times\ell \ } & S\otimes_k [A]_{\frac{e-1}{2}} \arrow{d}{\rotatebox{270}{$\sim \ $}}{\ \ \times\ell} \\%
S\otimes_k [A]_{\frac{e-1}{2}-1} \arrow{r}{B_{\frac{e-1}{2}-1}}&\to S\otimes_k [A]_{\frac{e+1}{2}}.
\end{tikzcd}\]
Then, choosing  appropriate bases, we can see $B_{\frac{e-1}{2}-2}$ as a submatrix of $B_{\frac{e-1}{2}-1}$.
Since $h_{\frac{e-1}{2}}=h_{\frac{e+1}{2}}$, $B_{\frac{e-1}{2}-1}$ and $B_{\frac{e-1}{2}-2}$ have the same number of rows, that is greater than or equal to the number of columns of $B_{\frac{e-1}{2}-1}$. This implies that each maximal minor of $B_{\frac{e-1}{2}-1}$ can be seen as a linear combination of maximal minors of $B_{\frac{e-1}{2}-2}$. Hence,
$$I(\C_{A,\frac{e-1}{2}-1})=(\text{max minors of } B_{\frac{e-1}{2}-1})\subseteq(\text{max minors of } B_{\frac{e-1}{2}-2})=I(\C_{A,\frac{e-1}{2}-2}).$$

 \emph{Case $e$ even.} By duality $I(C_{A,\frac{e}{2}-2})=I(C_{A,\frac{e}{2}})$. So it is enough to show 
 $$I\big(\C_{\frac{e}{2}-1}\big)\subseteq I\big(\C_{\frac{e}{2}}\big).$$
 Since we assume $A$ has the WLP we can choose a linear form  $\ell \in [R]_1$ such that the map
 $\times\ell:[A]_{\frac{e}{2}-1}\to [A]_{\frac{e}{2}}$ is injective and $\times\ell:[A]_{\frac{e}{2}+1}\to [A]_{\frac{e}{2}+2}$ is surjective. Then we have the commutative diagram
  \[ \begin{tikzcd}
S\otimes_k [A]_{\frac{e}{2}-1} \arrow{r}{B_{\frac{e}{2}-1}} \arrow[hook, swap]{d}{\times\ell \ }& S\otimes_k [A]_{\frac{e}{2}+1}   \arrow[two heads]{d}{\ \ \times\ell} \\%
S\otimes_k [A]_{\frac{e}{2}} \arrow{r}{B_{\frac{e}{2}}}& S\otimes_k [A]_{\frac{e}{2}+2}
\end{tikzcd}
\]
and, with appropriate choices of bases, the diagonal map 
can be represented by a matrix $B$ of linear forms in $S$ such that $B$ is a submatrix of $B_{\frac{e}{2}-1}$ and $B_{\frac{e}{2}}$ is a submatrix of $B$:
 \[ \begin{tikzcd}
S\otimes_k [A]_{\frac{e}{2}-1} \arrow{r}{B_{\frac{e}{2}-1}} \arrow[hook, swap]{d}{\times\ell \ }\arrow{dr}{B} & S\otimes_k [A]_{\frac{e}{2}+1}  \arrow[two heads]{d}{\ \ \times\ell} \\%
S\otimes_k [A]_{\frac{e}{2}} \arrow{r}{B_{\frac{e}{2}}}& S\otimes_k [A]_{\frac{e}{2}+2}.
\end{tikzcd}
\]
First we notice that $B_{\frac{e}{2}-1}$ is a square sub-matrix of $B$, with the same number of columns as $B$, so
$$I(\C_{A,\frac{e}{2}-1})=(\det B_{\frac{e}{2}-1})\subseteq (\text{max minors of } B).$$
On the other side $B$ is a submatrix of $B_{\frac{e}{2}}$, with the same number of rows, that is smaller or equal to $h_\frac{e}{2}$, the number of columns of $B_{\frac{e}{2}}$. Therefore each maximal minor of $B$ is a maximal minor of $B_{\frac{e}{2}}$ as well.
$$I(\C_{A,\frac{e}{2}-1})\subseteq (\text{max minors of } B)\subseteq(\text{max minors of } B_{\frac{e}{2}})=I(\C_{A,\frac{e}{2}}).$$
\end{proof}
As we can see in the following examples, having the Weak Lefschetz Property and admitting forms of degree $2$ such that the multiplication has maximum rank in each degree are independent properties.
\begin{ex}
    Consider $F=x_1x_2 (x_1^2x_3 - x_2^2x_4 ) \in \kk[x_1,x_2,x_3,x_4]$, and
    $$I= Ann(F)= \left(x_{4}^{2},\,x_{3}x_{4},\,x_{3}^{2},\,x_{1}^{2}x_{4},\,x_{2}^{2}x_{3},\,x_{1}^{2}x_{3}+x_{2}^{2}x_{4},\,x_{2}^{4},\,x_{1}^{2}x_{2}^{2},\,x_{1}^{4}\right).$$
 $R/I$ is a Gorenstein Algebra that fails the WLP as shown in  \cite[Example 3.7]{BinMaculayDual}.
 Using the computer algebra system  \texttt{Macaulay2} \cite{M2} we can show $I(C_A)\neq0$, for example the multiplication map $A_i\to A_{i+2}$ given by $(x_1+x_2+x_3+x_4)^2$ always has maximum rank.
\end{ex}
\begin{ex}
If we consider instead $F=x_1x_2^2(x_{1}^{2}x_{3}-x_{4}^{4}x_{4})$, we obtain the Gorenstein ideal
$$I= Ann(F)= \left(x_{4}^{2},\,x_{3}x_{4},\,x_{3}^{2},\,x_{1}^{2}x_{4},\,2\,x_{1}^{2}x_{3}+x_{2}^{2}x_{4},\,x_{2}^{3}x_{3},\,x_{1}^{4},\,x_{2}^{5},\,x_{1}^{2}x_{2}^{3}\right)$$
    Using  \texttt{Macaulay2} \cite{M2} we can show that the Gorenstein algebra has the WLP  but $I(C_A)=0$.
    In fact, for every $c\in R_2$ the multiplication map $\times c: A_2\to A_4$ is neither injective nor surjective. 
\end{ex}

\section{First cohomology module of a rank $2$ vector bundle over $\PP^2$}

In this section, we will focus on the case of three variables. It is easy to see that all the definitions introduced in the previous section for Artinian algebras can be extended to finite-length modules over the polynomial ring $R$.
For any finite length graded module $M$ over $R=\kk[x_1,x_2,x_3]$, we say that a form $C\in [R]_2$ is a \emph{Lefschetz conic}  if the multiplication map $\times C: [M]_i\to [M]_{i+2}$ has maximum rank in each degree, and we refer to $\C_M$ (Definition \ref{C}) with its structure as a subscheme of $\PP^5$ as \emph{the non-Lefschetz locus of conics} of $M$.

Our goal is to study the non-Lefschetz locus for a height three complete intersection $A=\kk[x_1,x_2,x_3]/(f_1,f_2,f_3)$, and as a generalization any finite length module $M$ that is the cokernel of a
graded map $$\varphi:\bigoplus_{i=1}^{n+2}R(-a_i)\to\bigoplus_{i=1}^{n} R(-b_i).$$

In \cite{FFP21}, the authors showed that  $M=\cok\varphi$ is isomorphic to the first cohomology module of a rank $2$ vector bundle $\E$ over $\PP^2$: $$M\cong\HH_*^1(\PP^2, \E):=\oplus_{t\in\Z}\HH^1(\PP^2, \E(t)).$$ Note that if $n=1$ and $b_1=0$ then the cokernel is a complete intersection $\kk[x_1,x_2,x_3]/(f_1,f_2,f_3)$ with $\deg f_i= a_i$.
These two descriptions are equivalent since any first cohomology module of a rank $2$ vector bundle over $\PP^2$ is of such form. This result is likely known to experts; however, since we were not able to find a proof in the literature, a proof is included for completeness.

\begin{prop}
    Let $\E$ be a vector bundle of rank $2$ over $\PP^2$ and let $M=\HH_*^1(\PP^2, \E)$. Assume $M\neq 0$. Then, there exist a finitely generate module $E$, with sheafification $\tilde{E}=\E$, and integers $n,a_i,b_i$ such that we have an exact sequence
    $$0\to E \to \bigoplus_{i=1}^{n+2}R(-a_i)\to\bigoplus_{i=1}^{n} R(-b_i)\to M\to 0\vspace{-5pt}.$$
\end{prop}
\begin{proof}
    Let $\E$ be a rank $2$ vector bundle, and let $M=\HH_*^1(\PP^2, \E):=\bigoplus_{t\in\Z}\HH^1(\PP^2, \E(t))$.
    $M$ is a finite length graded module over $R=\kk[x_1,x_2,x_3]$  \cite[Ch.III, \S 5-7]{Hartshorne}.
    We assumed $M\neq 0$, excluding the case where $\E$ splits.
    Let $E:=\Gamma_*(\PP^2, \E)=\bigoplus_{t\in\Z}\HH^0(\PP^2, \E(t))$. 
    
    $E$ is a finitely generated graded $R-$module and $\tilde{E}\cong\E$ \cite[Proposition III.5.15]{Hartshorne}.
    Using \cite[Theorem 13.21]{24h}, we have that $\HH^2_m(E)=\HH_*^1(\PP^2, \E)=M\neq 0$, hence $\depth E\leq 2$.
    Since $\Gamma_*(\PP^2, \Tilde{\E})=\Gamma_*(\PP^2, \E)=E$ by definition of $E$, the depth of $E$ must be exactly $2$ \cite[Theorem 13.22]{24h} and so its projective dimension is $1$. 
    Then $E$ has a free resolution of length $1$ (so with $2$ free terms), that we can sheafify and dualize to get the exact sequence of sheaves 
    $$0\to  \check{\E} \to \bigoplus_{i=1}^{n+2}\mathcal{O}_{\mathbb{P}^2}(\bar{a}_i)\to\bigoplus_{i=1}^{n} \mathcal{O}_{\mathbb{P}^2} (\bar{b}_i) \to 0$$
    for some integers $\bar{a}_i,\bar{b}_i$.
Since $\E$ is a rank $2$ vector bundle over $\PP^2$, it follows that  
$\check{\E}\cong \E(d)$, where $-d=c_1(\E)$ is the first Chern class of $\E$ \cite{Hartshorne2,OSS}.
Then shifting by $-d$ we obtain the exact sequence of sheaves 
$$0\to  \E \to \bigoplus_{i=1}^{n+2}\mathcal{O}_{\mathbb{P}^2}(-a_i)\to\bigoplus_{i=1}^{n} \mathcal{O}_{\mathbb{P}^2} (-b_i) \to 0$$
where we define $a_i=d-\bar{a}_i$, and $b_i=d-\bar{b}_i$.
Applying the functor $\Gamma_*$ we get the sequence

    $$0\to E\to \bigoplus_{i=1}^{n+2}R(-a_i)\to\bigoplus_{i=1}^{n} R(-b_i) \to \HH_*^1(\PP^2, \E)\to 0$$
 as desired.  
\end{proof}

Let $\E$ be a rank $2$ vector bundle over $\PP^2$ and $M=\HH_*^1(\PP^2, \E)$ be its first cohomology module.

\begin{rmk}
    A conic $C$ is a Lefschetz conic for $M=\HH_*^1(\PP^2, \E)$ if and only if \begin{align*}
    \times C: \equalto{[M]_i }{\HH^1(\PP^2, \E(i))} \to \equalto{[M]_{i+2}}{\HH^1(\PP^2, \E(i+2))}
\end{align*} has maximum rank for each $i$. Since this property is independent of the shift, we can always assume $\E$ is normalized, hence it has first Chern class $c_1(\E)\in\{0,-1\}$ (\cite{OSS}). With an abuse of notation we will label with $\times C$ the map $\HH^1(\PP^2, \E(i)) \to\HH^1(\PP^2, \E(i+2))$.
\end{rmk}

Before proceeding with the study of the non-Lefschetz locus of conics of $M=\HH_*^1(\PP^2, \E)$, we recall some results about vector bundles over projective space that we will use in the following sections.

\begin{df}[\cite{OSS}]
Let $\E$ be a rank $r$ vector bundle over $\PP^n$. Let $c_1(\E)$ be its first Chern class and $\mu(\E)=c_1(\E)/r$ the \emph{slope} of $\E$.
We say that 
\begin{itemize}
    \item $\E$ is \emph{semistable} if for any non-zero coherent subsheaf $\mathcal{F}\subset \E$ the slope satisfies $\mu(\mathcal{F})\leq \mu(\E)$;
    \item $\E$ is \emph{stable} if these inequalities are always strict, i.e.\ if for any non-zero coherent subsheaf $\mathcal{F}\subset \E$, we have $\mu(\mathcal{F})<\mu(\E)$;
    \item $\E$ is \emph{unstable} if it is not semistable.
\end{itemize}
\end{df}

Since we restrict our attention to the case of (normalized) rank $2$ vector bundles,  we use the following equivalent classification as the defining property for stability.
\begin{lemma}[\cite{OSS}]\label{stab}
    Let $\E$ be a normalized rank $2$ vector bundle over $\PP^n$. Then
\begin{itemize}
    \item $\E$ is stable if and only if it has no global sections, i.e. $\HH^0(\PP^n, \E)=0$;
     \item if $c_1(\E)=-1$, then stability and semistability are equivalent;
     \item if $c_1(\E)=0$, then $\E$ is semistable if and only if  $\HH^0(\PP^n, \E(-1))=0$. 
\end{itemize}
\end{lemma}

\begin{df}[\cite{FFP21}]
   For an unstable normalized rank $2$ vector bundle  over $\PP^n$ we can define the \emph{instability index} of $\E$ as the largest integer $k$ such that $\HH^0(\PP^n, \E(-k))\neq0$. \end{df}
As a consequence, for an unstable vector bundle $\E$
\begin{itemize}
    \item $k>0$ if $c_1(\E)=0$
    \item $k\geq0$ if $c_1(\E)=-1$.
\end{itemize}

Recall that every vector bundle over $\PP^1$ splits by Grothendieck's Theorem \cite{OSS}, hence, for a rank $2$ vector bundle over $\PP^n$, $\E_{|\ell}$ splits for any line $\ell$ in $\PP^n$.
\begin{df}[\cite{OSS}]
	Given a line $\ell$ the \emph{splitting type} of $\E$ on $\ell$ is the couple $(a,b)$ where $\E_{|\ell}\cong\mathcal{O}_{\mathbb{P}^1}(a)\oplus\mathcal{O}_{\mathbb{P}^1}(b)$.
\end{df}

\begin{teo}[Grauert--M\"{u}lich Theorem \cite{OSS}]
Let $\E$ be a semistable normalized rank $2$ vector bundle on $\PP^n$ and let $\ell$ be a general line. Then
	\[\E_{|\ell}\cong\begin{cases}
	    \mathcal{O}_{\mathbb{P}^1}\oplus\mathcal{O}_{\mathbb{P}^1} & \text{ if } c_1(\E)=0;\\
        \mathcal{O}_{\mathbb{P}^1}(-1)\oplus\mathcal{O}_{\mathbb{P}^1}& \text{ if } c_1(\E)=-1.
	\end{cases}\]
\end{teo}

We are mainly interested in rank $2$ vector bundles on $\PP^2$. The Grauert--M\"{u}lich Theorem gives us the splitting type of $\E$ over a general line $\ell$ when $\E$ is semistable. For $\E$ unstable we have:
\begin{teo}[\cite{FFP21}]\label{GM}Let $\E$ be an unstable normalized vector bundle of rank $2$ over $\PP^2$, and let $k$ be the instability index of $\E$. Then for a general line $\ell$ 
	\[\E_{|\ell}\cong\begin{cases}
	     \mathcal{O}_{\mathbb{P}^1}(-k)\oplus\mathcal{O}_{\mathbb{P}^1}(k)& \text{ if } c_1(\E)=0;\\
        \mathcal{O}_{\mathbb{P}^1}(-k-1)\oplus\mathcal{O}_{\mathbb{P}^1}(k) & \text{ if } c_1(\E)=-1.
	\end{cases}\]
\end{teo} 
\begin{df}
	A \emph{jumping line} is a linear element $\ell$ with splitting type different from the splitting type of a general line.
\end{df}

Finally, we have the following definition:
\begin{df}[\cite{jump}]\label{secondtype}
   For a semistable rank $2$ vector bundle $\E$ over $\PP^2$ with first Chern class odd, a line $\ell$ is called a \emph{jumping line of the second kind} for $\E$ when $\HH^0(\PP^2,\E_{| \ell^2})\neq0$.
\end{df}
Hulek in \cite{jump} proved the following results in this case:
\begin{teo}[\cite{jump}]\label{secondjump}
    Let $\E$ be a semistable rank $2$ vector bundle over $\PP^2$ with $c_1(\E)=-1$. Then the jumping lines of the second kind form a curve of degree $2(c_2(\E)-1)$.
\end{teo}
\begin{teo}[\cite{jump}]\label{thmjumping}
    Let $\E$ be a general semistable rank $2$ vector bundle over $\PP^2$ with $c_1(\E)=-1$ and second Chern class $c_2(\E)$. Then there are $\binom{c_2(\E)}{2}$ jumping lines, and they correspond to the singular points of the curve formed by the jumping lines of the second kind of $\E$.
\end{teo}
\begin{rmk}\label{moduli}
    We can talk about \emph{general vector bundle} since the moduli space of the rank $2$ vector bundles over $\PP^2$ with $c_1=-1$ and fixed second Chern class $c_2$ is irreducible by \cite{jump}. Similarly, also the moduli space of the rank $2$ vector bundles with $c_1=0$  and fixed second Chern class $c_2$ is irreducible (see \cite{OSS}).
\end{rmk}

\section{Strong Lefschetz property at range 2}

To study the non-Lefschetz locus of conics $\C_M\subset \PP^5$ we first want to show that it has positive codimension. Equivalently, we will show that there exists a Lefschetz conic. In fact, we will prove that there exists a linear form $\ell\in [R]_1$ such that the multiplication map 
$\times \ell^2: [M]_i\to [M]_{i+2}$ has maximum rank in each degree. 

\begin{prop}\label{LC}
    Let $M=\HH_*^1(\PP^2,\E)$ be the first cohomology module of $\E$, a rank $2$ vector bundle over $\PP^2$.  Then the multiplication map $\times \ell^2: [M]_i\to [M]_{i+2}$, for a general linear form $\ell\in [R]_1$, has maximum rank in each degree.
    Moreover, the set of lines $\ell$ that fail to have this property form a hypersurface in $(\PP^2)^*$.
\end{prop}
\begin{proof}
    Without loss of generality, we can assume $\E$ is normalized. A general line $\ell$ is a Lefschetz element for $M$ by \cite{FFP21}. In fact, by \cite{FFP21} we have:
\begin{itemize}
    \item if $\E$ is semistable and $c_1(\E)=0$, then $\times \ell: [M]_{i-1}\to [M]_{i}$ is injective for $i<-1$, bijective for $i=-1$, and surjective for $i >-1$;
    \item if $\E$ is unstable with index of instability $k$ and $c_1(\E)=0$, then $k>0$ and $\times \ell$ is injective for $i<-(k+1)$, bijective for $-(k+1) \leq i \leq k-1$, and surjective for $i >k-1$;
    \item if $\E$ is unstable with index of instability $k$ and $c_1(\E)=-1$, then $k\geq 0$ and $\times \ell$ is injective for $i<-(k+1)$, bijective for $-k-1\leq i\leq k$, and surjective for $i >k$.
\end{itemize}
    
    \emph{Case 1: $\E$ is unstable or $c_1(\E)=0$.} In this case, the multiplication map in the middle degree is an isomorphism. 
Then the multiplication by $\ell^2:[M]_i\to [M]_{i+2}$ must always have maximum rank for each $i$.
This also shows that if $\E$ is unstable or if the first Chern class is even, the lines $\ell$ for which $\times \ell^2: [M]_i\to [M]_{i+2}$ fails to have maximum rank in some degree are exactly the non-Lefschetz elements of $M$. These are the jumping lines of $\E$ by \cite{vb}, which, in this case, form a hypersurface in $(\PP^2)^*$.

 \emph{Case 2: $\E$ is semistable with $c_1(\E)=-1$.}
In this case by \cite{FFP21} the multiplication map $\times \ell: [M]_{i-1}\to [M]_{i}$ by a general line $\ell$ is injective for $i\leq -1$ and surjective for $i\geq 0$. Therefore $\times \ell^2: [M]_{i}\to [M]_{i+2}$, when $\ell$ is a general linear form, is injective for $i<-2$ and surjective for $i\geq -1$. So it is enough to show that there exists $\ell\in [R]_1$ such that \begin{align*}
    \times \ell^2: \equalto{[M]_{-2}}{\HH_*^1(\PP^2, \E(-2))} \to \equalto{[M]_{0}}{\HH_*^1(\PP^2, \E)}
\end{align*} has maximum rank.
From the short exact sequence
    $$0\to\E(-2)\stackrel{\times \ell^2}{\to}\E\to\E_{| \ell^2}\to 0$$
we get the long exact sequence 
\begin{align*}
	0\to \HH^0(\PP^2,\E(-2))\to \HH^0(\PP^2,\E)\to \HH^0(\PP^2,\E_{| \ell^2})\to \HH^1(\PP^2,\E(-2))\stackrel{\times \ell^2}{\to} \HH^1(\PP^2,\E)\to \\
	\to \HH^1(\PP^2,\E_{| \ell^2})\to \HH^2(\PP^2,\E(-2))\to \HH^2(\PP^2,\E)\to \HH^2(\PP^2,\E_{| \ell^2})=0.
\end{align*}
Since $\E$ is semistable and $c_1(\E)=-1$, $\HH^0(\PP^2,\E)=0$. Using duality 
$$\HH^2(\PP^2,\E(-2))\cong \HH^0(\PP^2,\check{\E}(-1))\cong \HH^0(\PP^2,\E)=0$$ since $\check{\E}=\E(1)$ for a rank $2$ vector bundle with $c_1(\E)=-1$.
Then  the sequence above becomes 
\begin{align*}
	0\to  \HH^0(\PP^2,\E_{| \ell^2})\to \HH^1(\PP^2,\E(-2))\stackrel{\times \ell^2}{\to} \HH^1(\PP^2,\E)
	\to \HH^1(\PP^2,\E_{| \ell^2})\to 0.
\end{align*}
It is sufficient that $\HH^0(\PP^2,\E_{| \ell^2})=0$ to obtain that $\times \ell^2: [M]_{-2}\to [M]_{0}$ is injective, and so $\ell^2$ has maximum rank. 
This condition is also necessary. In fact, using duality again we have that $\HH^1(\PP^2,\E(-2))\cong \HH^1(\PP^2,\check{\E}(-1))\cong \HH^1(\PP^2,\E)$, hence $\times \ell^2: [M]_{-2}\to [M]_{0}$ has maximum rank if and only if it is an isomorphism.

So we conclude that $\times \ell^2: [M]_i\to [M]_{i+2}$ fails to have maximum rank if and only if $\HH^0(\PP^2,\E_{| \ell^2})\neq0$, i.e. if and only if $\ell$ is a jumping line of $\E$ of the second kind (Definition \ref{secondtype}). By Theorem \ref{secondjump} these form a hypersurface in $(\PP^2)^*$.
 In particular, the complement is not empty, so for a general line $\ell$ the multiplication map $\times \ell^2:[M]_i\to [M]_{i+2}$ has maximum rank in each degree.
\end{proof}

\begin{cor}\label{SLP2}
  Any Artinian complete intersection $A=R/(f_1,f_2,f_3)$,  has the Strong Lefschetz Property at range $2$, i.e. there exists a linear form $\ell\in [R]_1$ such that multiplication map $\times \ell^2: [A]_i\to [A]_{i+2}$ has maximum rank in each degree. 
\end{cor}
\begin{proof}
    The first syzygy bundle $\E$ of $A$ is a rank $2$ vector bundle over $\PP^2$ and $A\cong \HH_*^1(\PP^2,\E)$, so we can apply Proposition \ref{LC}.
\end{proof}

\begin{rmk}\label{singular}
    The proof of Proposition \ref{LC} shows that if $\E$ is unstable or if the first Chern class of $\E$ is even, then $\times \ell^2: [M]_i\to [M]_{i+2}$ fails to have maximum rank in some degree if and only if $\ell$ is a jumping line. The same proof shows that $\times \ell_1\ell_2: [M]_i\to [M]_{i+2}$ fails to have maximum rank in some degree if and only if at least one of $\ell_1$ and $\ell_2$ is a jumping line. The same is not true for the case when $\E$ is semistable with the first Chern class odd. 
\end{rmk}

\section{Jumping conics and non-Lefschetz conics}

In this section, we want to classify the elements of the non-Lefschetz locus of conics of $M$, the first cohomology module of a rank $2$ vector bundle over $\PP^2$. In the case of lines, the non-Lefschetz elements are exactly the jumping lines by \cite{vb}. Hence, the first question that we want to address is whether the elements of the non-Lefschetz locus of conics are exactly the jumping conics. Unfortunately, this is not always the case. We will see that every non-Lefschetz conic is a jumping conic, but the converse is not always true. 

We first recall the notion of a jumping conic for a semistable vector bundle $\E$ of rank $2$ over $\PP^2$, introduced by Vitter in \cite{Vitter}.
We can assume that $\E$ is normalized. If $C$ is a smooth conic, then $\E$ splits over $C$ as
$$\E_{|C}\cong\begin{cases}
		\OO_{\PP^1}(a)\oplus\OO_{\PP^1}(-a)  & \text{ if }  c_1(\En)=0; \\
		\OO_{\PP^1}(a)\oplus\OO_{\PP^1}(-a-2) &\text{ if }  c_1(\En)=-1.
	\end{cases} $$

Vitter in \cite{Vitter} generalizes the Grauert--M\"{u}lich Theorem, and proves the following:
\begin{teo}{\emph{(}\cite[Corollary 1]{Vitter}\emph{)}} 
	Let $\E$ be a semistable normalized rank $2$ vector bundle on $\PP^2$ and let $C$ be a general smooth conic. Then
	$$\E_{|C}\cong\begin{cases}
		\OO_{\PP^1}\oplus\OO_{\PP^1}  & \text{ if }  c_1(\E)=0; \\
		\OO_{\PP^1}(-1)\oplus\OO_{\PP^1}(-1) &\text{ if }  c_1(\E)=-1.
	\end{cases} $$
\end{teo}
Hence a smooth jumping conic for a semistable vector $\E$ bundle is defined as a smooth conic for which the splitting type differs from the one described in the previous result.
\begin{df}[\cite{Vitter}]\label{splitstable}
	A smooth conic $C$ is a jumping conic for a semistable vector bundle $\E$  if 
	$$\E_{|C}\cong\begin{cases}
		\OO_{\PP^1}(a)\oplus\OO_{\PP^1}(-a)  & \text{ if }  c_1(\En)=0; \\
		\OO_{\PP^1}(a-1)\oplus\OO_{\PP^1}(-a-1) &\text{ if }  c_1(\En)=-1
	\end{cases} $$
	with $a>0$.
\end{df}

To extend this definition also for unstable vector bundles, first we need to study how an unstable vector bundle $\E$ splits when restricted to a general smooth conic.

 \begin{prop}
 	Let $\E$ be an unstable normalized rank $2$ vector bundle on $\PP^2$, with instability index $k$, and let $C$ be a general smooth conic. Then
 	$$\E_{|C}\cong\begin{cases}
 		\OO_{\PP^1}(2k)\oplus\OO_{\PP^1}(2k)  & \text{ if }  c_1(\E)=0; \\
 		\OO_{\PP^1}(2k)\oplus\OO_{\PP^1}(-2k-2) &\text{ if }  c_1(\E)=-1.
 	\end{cases} $$
 \end{prop} 
 \begin{proof}
 	Let $\E$ be an unstable normalized rank $2$ vector bundle on $\PP^2$, with instability index $k$, and let $C$ be a general smooth conic.
 	We know that $h^0(\PP^2,\E(-k))\neq 0$, and any non-zero section $s\in \HH^0(\PP^2,\E(-k))$ is regular, so its vanishing locus has codimension at least $2$. Then a
 	 non-zero section $s\in \HH^0(\PP^2,\E(-k))$ gives us the exact sequence
 	$$0\to \OO_{\PP^2}\to \E(-k)\to \mathcal{I}\to 0$$
 	 where $\mathcal{I}$ is the ideal sheaf of a set of points. Then shifting by $k$ we get the sequence
 	$$0\to \OO_{\PP^2}(k)\to \E\to \mathcal{Q}\to 0$$
 	where we define $\mathcal{Q}=\mathcal{I}(k)$.
 	Since $C$ is general, we can assume $C$ does not meet the zero locus of $s$, therefore restricting the previous sequence to $C$, we have the exact sequence
 	$$0\to \OO_{\PP^2|C}(k)\to \E_{|C}\to \mathcal{Q}_{|C}\to 0.$$ 
 	$C$ is a smooth conic so $\OO_{\PP^2|C}(k)\cong \OO_{\PP^1}(2k)$ and $\mathcal{Q}_{|C}$ is a line bundle over $C\cong \PP^1$, so there exists an $\ell$ such that $\mathcal{Q}_{|C}\cong \OO_{\PP^1}(\ell)$.
 	Then our sequence becomes 
 	$$0\to \OO_{\PP^1}(2k)\to \E_{|C}\to \OO_{\PP^1}(\ell)\to 0.$$
 	We can use this sequence to compute the Euler characteristic of $\E_{|C}$
 	$$\chi(\E_{|C})=\chi(\OO_{\PP^1}(2k))+\chi(\OO_{\PP^1}(2k))=2k+l+2.$$
 	Comparing with
 	$$\chi(\E_{|C})=\begin{cases}
 		2  & \text{ if }  c_1(\En)=0; \\
 		0 &\text{ if }  c_1(\En)=-1,
 	\end{cases} $$
 	we obtain
 	$$\ell=\begin{cases}
 		-2k  & \text{ if }  c_1(\En)=0; \\
 		-2k-2 &\text{ if }  c_1(\En)=-1.
 	\end{cases} $$
 	In both cases, the sequence 
 	$$0\to \OO_{\PP^1}(2k)\to \E_{|C}\to \OO_{\PP^1}(\ell)\to 0$$
 	is a split exact sequence since $\Ext(\OO_{\PP^1}(2k),\OO_{\PP^1}(\ell))=0$, so
 	$$\E_{|C}\cong\begin{cases}
 		\OO_{\PP^1}(2k)\oplus\OO_{\PP^1}(2k)  & \text{ if }  c_1(\En)=0; \\
 		\OO_{\PP^1}(2k)\oplus\OO_{\PP^1}(-2k-2) &\text{ if }  c_1(\En)=-1.
 	\end{cases} $$
 \end{proof}
 We extend Definition \ref{splitstable} for any rank $2$ vector bundle on $\PP^2$ as follows:
 \begin{df}
 	A smooth \emph{jumping conic} for $\E$ is a smooth conic $C$ such that the restriction of $\E$ to $C$ does not split as it does for a general conic.
 \end{df}

Until now, we have considered only smooth conics. The next question is how to extend the definition of jumping conic to include also the singular case. When $\E$ is semistable, we recall the definition in \cite{Vitter}:
\begin{df}[\cite{Vitter}]\label{defsing}
Let $\E$ be a rank $2$ normalized semistable bundle. A singular conic $C=\ell_1\ell_2$ is a jumping conic if 
\begin{itemize}
    \item $h^0(\C;\E_{|C}) > 0$ when $c_1(\E)=-1$;
    \item if either $\ell_1$ or $\ell_2$ is a jumping line when $c_1(\E)=0$.
\end{itemize} 
\end{df}
\begin{rmk}
    In the case $c_1(\E)=-1$, this definition agrees with the one for smooth jumping conics; in fact, we can say that a conic $C$ is a jumping conic if and only if $h^0(\C;\E_{|C}) > 0$.
\end{rmk}
Vitter also shows that for singular conics Definition \ref{defsing} is equivalent to the following:
\begin{df} $C=\ell_1\ell_2$  is a jumping conic for a semistable normalized vector bundle $\E$
exactly when one of the following is true:
\begin{itemize}
	\item $\ell_1$ or $\ell_2$ is a jumping line;
	\item $\C = \ell^2$ and $\ell$ is a jumping line
	of the second kind (Definition \ref{secondtype});
	\item $\ell_1$ and $\ell_2$ are generic so that $\E_{|\ell_j}=\OO_{\PP^1}\oplus \OO_{\PP^1}(-1)$ for
	$j = 1, 2$ and the $\OO_{\PP^1}$ summands coincide at the intersection point.
\end{itemize} 
\end{df}
We can extend this definition to the case when $\E$ is unstable in the following natural way:
\begin{df}
Let $\E$ be a rank $2$ normalized unstable vector bundle. A singular conic $C=\ell_1\ell_2$ is a jumping conic if $\ell_1$ or $\ell_2$ is a jumping line.
\end{df}

To connect the notion of jumping conic with the non-Lefschetz conics let us first state some facts that we will use in the next proof. Recall that a conic $\C$ is a  Lefschetz conic for $M=\HH_*^1(\PP^2, \E)$ if and only if \begin{align*}
    \times C: \equalto{[M]_i }{\HH_*^1(\PP^2, \E(i))} \to \equalto{[M]_{i+2}}{\HH_*^1(\PP^2, \E(i+2))}
\end{align*} has maximum rank for each $i$. Since all these properties are independent of the shift, we can always assume $\E$ is normalized.

With a similar argument to the one in the proof of Theorem \ref{LC}, from the short exact sequence
$$0\to\E(i-2)\to\E(i)\to\E(i)_{| C}\to0.$$
we have the long exact sequence 
\begin{align*}
	0&\to \HH^0(\PP^2,\E(i-2))\to \HH^0(\PP^2,\E(i))\to \HH^0(\PP^2,\E(i)_{| C})\to\\
 &\to \HH^1(\PP^2,\E(i-2))\stackrel{\times  C}{\to} \HH^1(\PP^2,\E(i))
	\to \HH^1(\PP^2,\E(i)_{| C})\to \\&\to \HH^2(\PP^2,\E(i-2))\to \HH^2(\PP^2,\E(i))\to \HH^2(\PP^2,\E(i)_{| C})=0.
\end{align*}\label{seq}
Then we have the following facts:
\begin{enumerate}
	\item \label{'A1} $\times  C$ is injective if $h^0(\PP^2,\E(i)_{| C})=0$ (sufficient condition but not necessary);
	\item \label{'A2} $\times  C$ is injective if and only if $h^0(\PP^2,\E(i)_{| C})- h^0(\PP^2,\E(i))+h^0(\PP^2,\E(i-2))=0$;
	\item \label{'A3} $\times  C$ is injective if and only if the map $  \HH^0(\PP^2,\E(i))\to \HH^0(\PP^2,\E(i)_{| C})$ is surjective;
	\item \label{'B1} $\times  C$ is surjective if $h^1(\PP^2,\E(i)_{| C})=0$ (sufficient condition but not necessary);
	\item \label{'B2} $\times  C$ is surjective if and only if $h^1(\PP^2,\E(i)_{| C})- h^2(\PP^2,\E(i-2))+h^2(\PP^2,\E(i))=0$;
	\item \label{'B3} $\times  C$ is surjective if and only if the map $  \HH^1(\PP^2,\E(i)_{| C})\to \HH^2(\PP^2,\E(i-2))$ is injective.
\end{enumerate}
From this point forward we treat the cases of $\E$ unstable and $\E$ stable separately.

\begin{prop}\label{jum1}
	Let $\E$ be an unstable vector bundle. The conic $C$ is not a Lefschetz conic if and only if it is a jumping conic.
\end{prop}
\begin{proof} As we saw in Remark \ref{singular}  a singular conic $C=\ell_1\ell_2$ is a non-Lefschetz conic if and only if at least one between $\ell_1$ or $\ell_2$ is a jumping line, i.e. if and only if $C$ is a jumping conic by definition.
Assume $C$ is a smooth conic and, without loss of generality, $\E$ is normalized. 
	
	For $\E$ unstable with instability index $k$ we know that 
	$$\E_{| C}\cong\begin{cases}
		\OO_{\PP^1}(a)\oplus\OO_{\PP^1}(-a)  & \text{ if }  c_1(\E)=0; \\
		\OO_{\PP^1}(a)\oplus\OO_{\PP^1}(-a-2) &\text{ if }  c_1(\E)=-1,
	\end{cases} $$ and $ C$ is a jumping conic if and only if $a\neq 2k$. 

	Let us first prove that for any smooth conic $ C$, we must have $a\geq 2k$.
	From the short exact sequence 
	$$0\to\E(-k-2)\to\E(-k)\to\E(-k)_{| C}\to0,$$
	we  have the cohomology sequence 
	$$0\to \HH^0(\PP^2,\E(-k-2))\to \HH^0(\PP^2,\E(-k))\to \HH^0(\PP^2,\E(-k)_{| C})\to\dots$$
	By definition of instability index we have that $\HH^0(\PP^2,\E(-k))\neq0$ and, for any $b>k$, $\HH^0(\PP^2,\E(-b))=0$. In particular $\HH^0(\PP^2,\E(-k-2))=0$, so the map $\HH^0(\PP^2,\E(-k))\to \HH^0(\PP^2,\E(-k)_{| C})$ is injective.
	It follows that  $\HH^0(\PP^2,\E(-k)_{|C})\neq 0$. Since
	$$\HH^0(\PP^2,\E(-k)_{| C})=\begin{cases}
		\HH^0(\PP^1,\OO_{\PP^1}(-2k+a)\oplus\OO_{\PP^1}(-2k-a) ) & \text{ if }  c_1(\E)=0; \\
		\HH^0(\PP^1,\OO_{\PP^1}(-2k+a)\oplus\OO_{\PP^1}(-2k-a-2)) &\text{ if }  c_1(\E)=-1,
	\end{cases} $$
	we can conclude that $a\geq 2k$. 
	
	\emph{Case 1: $c_1(\E)=0$.} 
	In this case $k>0$ and $\E^\vee=\E$.
	Using \cite[Proposition 3.6]{FFP21}  we have for $i<k$
	\begin{align*}
		&h^0(\PP^2,\E(i)_{|\ell})- h^0(\PP^2,\E(i))+h^0(\PP^2,\E(i-2))\\
		&=h^0(\PP^2,\mathcal{O}_{\mathbb{P}^1}(2i-a))+ h^0(\PP^2,\mathcal{O}_{\mathbb{P}^1}(2i+a))-\binom{k+i+2}{2}+\binom{k+i}{2}\\
		&=\begin{cases}
			4i+2  & \text{ if }  i\geq \frac{a}{2};\\
			a+2i+1  & \text{ if } -\frac{a}{2}\leq i<\frac{a}{2};\\
			0  & \text{ if } i<-\frac{a}{2};
		\end{cases}+\begin{cases}
			-2k-2i-1 & \text{ if } -k\leq i<k;\\
			0  & \text{ if } i<-k;
		\end{cases} \\
		&=\begin{cases}
			a-2k  & \text{ if }  -k\leq i<k;\\
			a+2i+1  & \text{ if } -\frac{a}{2}\leq i<-k;\\
			0  & \text{ if } i<-\frac{a}{2}.
		\end{cases}
	\end{align*}

	If $a=2k$ then $h^0(\PP^2,\E(i)_{|\ell})- h^0(\PP^2,\E(i))+h^0(\PP^2,\E(i-2))=0$ for every $i<k$ and so by Fact (\ref{'A2}), the map $\times  C$ is injective for $i<k.$
 
	When $a>2k$, using Fact (\ref{'A2})  we get instead that  $\times  C$ is not injective for $-\frac{a}{2}\leq i<k$ (and it  is injective for $i< -\frac{a}{2}$).
	
	In a similar way using Serre Duality, and \cite[Proposition 3.7 ]{FFP21}, for $i\geq-k$
	\begin{align*}
		&h^1(\PP^2,\E(i)_{| C})- h^2(\PP^2,\E(i-2))+h^2(\PP^2,\E(i))=\\
		&h^1(\PP^2,\mathcal{O}_{\mathbb{P}^1}(2i-a))+ h^1(\PP^2,\mathcal{O}_{\mathbb{P}^1}(2i+a))- h^0(\PP^2,\E(-i-3))+h^0(\PP^2,\E(-i-1))=\\
		&h^0(\PP^2,\mathcal{O}_{\mathbb{P}^1}(-2i+a-2))+ h^0(\PP^2,\mathcal{O}_{\mathbb{P}^1}(-2i-a-2))-\binom{k-i+1 }{2}-\binom{k-i-1}{2}\\
		&= \begin{cases}
			0& \text{ if }  i>\frac{a}{2}-1;\\
			-2i+a-1  & \text{ if } -\frac{a}{2}-1< i\leq \frac{a}{2}-1;\\
			-4i-2  & \text{ if } i\leq-\frac{a}{2}-1;
		\end{cases}+\begin{cases}
			0 & \text{ if } i\geq k;\\
			2i-2k+1 & \text{ if }-k\leq i<k .
		\end{cases}
	\end{align*}
		If $a=2k$,	using Fact (\ref{'B2}),   $\times  C$ is surjective for   $i\geq-k$, so in this case  $\times  C$  always has maximum rank.
	If $a>2k$, $\times C$ is  not surjective for $-k\leq i< \frac{a}{2}$. So  $\times  C$ does not have maximum rank for $-k \leq i<k$; since $k>0$ this interval is not empty. Hence a smooth conic $C$ is not a Lefschetz conic for $\E$, when $a>2k$.
	
This proves that for  $\E$ unstable, normalized with $c_1(\E)=0$, $C$ is a Lefschetz conic if and only if $C$ is not a jumping conic.

	\emph{Case 2: $c_1(\E)=-1$}. In this case $\E$ is unstable and normalized with $c_1(\E)=-1$, therefore $k\geq0$ and $\E^\vee=\E(1)$. We proceed in a similar way to Case 1.	
 
 Using \cite[Proposition 3.7 ]{FFP21}  we have 
	\begin{align*}
		&h^0(\PP^2,\E(i)_{| C})- h^0(\PP^2,\E(i))+h^0(\PP^2,\E(i-2))\\
		&=h^0(\PP^2,\mathcal{O}_{\mathbb{P}^1}(2i+a)\oplus\mathcal{O}_{\mathbb{P}^1}(2i-a-2))- h^0(\PP^2,\E(i))+h^0(\PP^2,\E(i-2))\\
		&=\begin{cases}
			4i  & \text{ if }  i> \frac{a}{2};\\
			2i+a+1  & \text{ if } -\frac{a}{2}\leq i\leq\frac{a}{2};\\
			0  & \text{ if } i<-\frac{a}{2};
		\end{cases}+
		\begin{cases}
			-2k-2i-1 & \text{ if } -k\leq i\leq k;\\
			0  & \text{ if } i<-k;
		\end{cases} \\
		&=\begin{cases}
			a-2k  & \text{ if }  -k\leq i\leq k;\\
			2i+a+1  & \text{ if } -\frac{a}{2}\leq i<-k;\\
			0  & \text{ if } i<-\frac{a}{2}.
		\end{cases}\\
	\end{align*}
	Using Serre Duality and \cite[Proposition 3.7]{FFP21} again, we have for $i\geq -k$

	\begin{align*}
		&h^1(\PP^2,\E(i)_{| C})- h^2(\PP^2,\E(i-2))+h^2(\PP^2,\E(i))=\\
		&h^0(\PP^2,\mathcal{O}_{\small{\mathbb{P}^1}}(-2i-a-2))\hspace{-2pt}+h^0(\PP^2,\mathcal{O}_{\small{\mathbb{P}^1}}(-2i+a))
  \hspace{-3pt}-h^0(\PP^2,\E(-i))\hspace{-2pt}+h^0(\PP^2,\E(-i-2))\\
		&= \begin{cases}
			0& \text{ if }  i>\frac{a}{2};\\
			-2i+a+1  & \text{ if } -\frac{a}{2}\leq i\leq\frac{a}{2};\\
			-4i  & \text{ if } i<-\frac{a}{2};
		\end{cases}+\begin{cases}
			0 & \text{ if } i\geq k;\\
			-2i-2k-1 & \text{ if }-k\leq i< k ;
		\end{cases}\\ 
		&=\begin{cases}
			0& \text{ if }  i>\frac{a}{2};\\
			2i-a-1  & \text{ if }  k<i\leq\frac{a}{2};\\
			a-2k  & \text{ if } -k\leq i\leq k.
		\end{cases}
	\end{align*}
We can conclude, using Fact (\ref{'A2}) and Fact (\ref{'B2}), that 
  \begin{itemize}
      \item  if $a=2k$, then $\times  C$ is always injective for $i\leq k$ and  surjective for any $i\geq -k$.  So $\times  C$ always has maximum rank.
      \item  $\times  C$ does not have maximum rank for $-k\leq i\leq k$, and since $k\geq 0$, this interval is not empty. So, in this case, $ C$ is not a Weak Lefschetz element.
  \end{itemize}
This concludes our proof: $ C$ is a Lefschetz conic if only if it is not a jumping conic.
\end{proof}

Let us now consider the semistable case.
When $\E$ is semistable with first Chern class $c_1(\E)=-1$, as noted by Vitter \cite{Vitter}, the jumping conics (smooth or singular) are exactly the ones for which $h^0(\C;\E_{|C}) \neq0$, and equivalently $h^1(\C;\E_{|C}) \neq0$, while this is not true if $c_1(\E)=0$. This difference is noticed in the proof of \cite[Theorem 2]{Vitter}. 
This is exactly where the difference between jumping conics and non-Lefschetz conics lies: while for $c_1(\E)=-1$ they are equivalent, for $c_1(\E)=0$ the jumping conics are a subset corresponding to the condition of  $h^1(\C;\E_{|C}) \neq0$. 

\begin{teo} \cite[Theorem 2]{Vitter} \label{vit}
	The set of jumping conics $J_2$ of a semistable rank $2$ vector bundle $\E$ on $\PP^2$ can be given the scheme structure of a hypersurface in $\PP^5$ of degree $c_2(\E)$ if
	$c_1(\E)=0$ and of degree $c_2(\E)-1$ if $c_1(\E)=-1$. Furthermore, the singular jumping
	conics are in the scheme-theoretic closure of the smooth jumping conics.	
\end{teo}
The following results can be seen as corollaries of this theorem. 

\begin{cor}\label{jum2}
	Let $\E$ be a stable, normalized vector bundle with $c_1(\E)=-1$. A conic $C$ fails to be a Weak Lefschetz conic if and only if it is a jumping conic.
\end{cor}
\begin{proof}
Let us first consider the case when $C=\ell_1\ell_2$ is singular. If $\C=\ell^2$ we show in the proof of Proposition \ref{LC} that $C$ is a non-Lefschetz conic if and only if it is a jumping line of the second type, i.e. $h^0(\C;\E_{|C}) \neq0$. The same proof shows that $C=\ell_1\ell_2$ is a non-Lefschetz conic if and only if $h^0(\C;\E_{|C}) \neq0$.

Let us assume now that $C$ is smooth.
	Then $\E_{| C}\cong	\OO_{\PP^1}(a-1)\oplus\OO_{\PP^1}(-a-1)$ and $ C$ is a jumping conic if and only if $a>0$.
 In this case $\chi(\E_{| C})=0$, so $h^0(\PP^2,\E_{| C})=h^1(\PP^2,\E_{| C})$, in particular for a jumping conic they are both non-zero. 
 
	Let us assume $C$ is a jumping conic. We have $\HH^0(\PP^2,\E)= 0$ because $\E$ is stable. Then the map $0= \HH^0(\PP^2,\E)\to \HH^0(\PP^2,\E_{| C})\neq0$ is not surjective, so by Fact (\ref{'A3}) the map $\times  C$ is not injective.
	
	The map $  \HH^1(\PP^2,\E_{| C})\to \HH^2(\PP^2,\E(-2))$ cannot be injective because  $\HH^1(\PP^2,\E_{| C})\neq 0$ and $\HH^2(\PP^2,\E(-2))\cong \HH^0(\PP^2,\check{\E}(1))\cong \HH^0(\PP^2,\E(-2))=0$, using Serre Duality and stability of $\E$. Then $\times  C$ is not surjective by Fact (\ref{'B3}). 
 In conclusion, $ C$ is not a Lefschetz conic because the multiplication map $\times C$ fails to have maximal rank from degree $-2$ to degree $0$.
 
	Let us consider now a smooth non jumping conic $ C$ so $\E_{| C}\cong	\OO_{\PP^1}(-1)\oplus\OO_{\PP^1}(-1)$. 
 Using Fact (\ref{'A1}), the map $\times  C$ is injective if $h^0(\PP^2,\E(i)_{| C})=h^0(\PP^1,\OO_{\PP^1}(2i-1)\oplus\OO_{\PP^1}(2i-1))=0$, then it is injective for every $i\leq0$.
	Now, $\times  C$ is surjective if $h^1(\PP^2,\E(i)_{| C})=0$ by Fact (\ref{'B1}).
	Using Serre Duality again, 
	\begin{align*}
		h^1(\PP^2,\E(i)_{| C})&= h^1(\PP^1,\OO_{\PP^1}(2i-1)\oplus\OO_{\PP^1}(2i-1))\\
		&= h^0(\PP^1,\OO_{\PP^1}(-2i-1)\oplus\OO_{\PP^1}(-2i-1)),
	\end{align*} 
	then it is zero for $i>0$.
	Then  $\times  C$ always has maximum rank, so $ C$ is a Lefschetz element.
\end{proof}

\begin{cor}\label{jum3}
	Let $\E$ be a semistable, normalized vector bundle with $c_1(\E)=0$. A smooth conic $C$ fails to be a Weak Lefschetz conic if and only if it is a jumping conic and $\E_{| C}\cong \OO_{\PP^1}(a)\oplus\OO_{\PP^1}(-a)$ with $a>1$. When $C$ is singular, the definitions of jumping and non-Lefschetz are equivalent.
\end{cor}
\begin{proof}
The case when $C$ is singular is the same as in Proposition \ref{jum1} for unstable vector bundles. In fact $C=\ell_1\ell_2$ is a non-Lefschetz conic if and only if at least one between $\ell_1$ or $\ell_2$ is a jumping line by Remark \ref{singular} as in the definition of singular jumping conic.

Let us assume first that $\E_{| C}\cong \OO_{\PP^1}(a)\oplus\OO_{\PP^1}(-a)$ with $a=0$ or $a=1$. We want to show that in this case, $C$ is a Lefschetz conic. 
	Since $h^0(\PP^2,\E(i)_{| C})=h^0(\PP^1,\OO_{\PP^1}(2i+1)\oplus\OO_{\PP^1}(2i-1))=0$,  for every $i<0$,  then Fact (\ref{'A1}) implies $\times  C$ injective for every $i<0$.
	
	The map $\times  C$ is surjective if $h^1(\PP^2,\E(i)_{| C})=0$ by Fact \ref{'B1}.
	Using Serre Duality 
	\begin{align*}
		h^1(\PP^2,\E(i)_{| C})&=h^1(\PP^1,\OO_{\PP^1}(2i+a)\oplus\OO_{\PP^1}(2i-a))\\
		&=h^0(\PP^1,\OO_{\PP^1}(-2i-2-a)\oplus\OO_{\PP^1}(-2i-2+a))
	\end{align*} 
	then it is zero for $i\geq0$ (here $a=0$ or $a=1$).
	Then  $\times  C$ always has maximal rank, so $ C$ is a Lefschetz element.
	
	Now we will show that for any smooth conic $C$ such that $\E_{| C}\cong \OO_{\PP^1}(a)\oplus\OO_{\PP^1}(-a)$ with $a>1$, $\times  C: \HH^1(\PP^2, \E(-2))\to\HH^1(\PP^2, \E)$ does not have maximal rank and so $ C$ is not a Lefschetz conic.
    Note that if $a>1$, $\HH^0(\PP^2\E_{| C})\cong \HH^0(\PP^2,	\OO_{\PP^1}(-a)\oplus\OO_{\PP^1}(a))\neq 0$ and, using Serre Duality
    $$\HH^1(\PP^2,\E_{| C})=\HH^1(\PP^1,\OO_{\PP^1}(-a)\oplus\OO_{\PP^1}(a))=\HH^0(\PP^1,\OO_{\PP^1}(a-2)\oplus\OO_{\PP^1}(-a-2))\neq 0.$$
 
	We consider first the case where $\E$ is stable; later we return to the case when $\E$ is semistable but not stable. 
	By stability $\HH^0(\PP^2,\E)=0$, and so $\HH^0(\PP^2,\E)\to \HH^0(\PP^2,\E_{| C})\neq0$ cannot be surjective. Therefore, by Fact (\ref{'A3}), the map $\times  C$ is not injective. 
	Using Serre Duality, and (semi)stability  we also have $$\HH^2(\PP^2,\E(-2))=\HH^0(\PP^2\E(-1))=0.$$
	Then the map $0\neq \HH^1(\PP^2,\E_{| C})\to \HH^2(\PP^2,\E(i-2))=0$ is not injective. Then $\times  C$ cannot be surjective by Fact (\ref{'B3}).

 	When $\E$ is semistable but not stable, the same argument shows that $\times  C$ is not surjective, but in this case $\HH^0(\PP^2,\E)\neq0$. By Fact (\ref{'A3}), the map $\times  C$  is injective if and only if the map $\HH^0(\PP^2,\E)\to \HH^0(\PP^2,\E_{| C})$ is surjective. Looking at the long exact sequence 
	\begin{align*}
		\HH^0(\PP^2,\E(-2))\to \HH^0(\PP^2,\E)\to \HH^0(\PP^2,\E_{| C})\to \HH^1(\PP^2,\E(-2))\stackrel{\times  C}{\to} \HH^1(\PP^2,\E)
		\to \cdots
	\end{align*}
	we know that that map is injective because $\HH^0(\PP^2,\E(-2))=0$. So $\times  C$  is injective if and only if the map $\HH^0(\PP^2,\E)\to \HH^0(\PP^2,\E_{| C})$ is an isomorphism, but this is not possible since $h^0((\PP^2,\E_{| C}))=a+1>2$ while we will show that $h^0(\PP^2,\E)=1.$

 We know that $h^0(\PP^2,\E)\neq 0$, so we can take  a non-zero section $s\in \HH^0(\PP^2,\E)$.
 	Since the section $s$ must be regular, we have the exact sequence
 	$$0\to \OO_{\PP^2}\to \E\to \mathcal{I}\to 0$$
 	where $\mathcal{I}$ is the ideal sheaf of a set of points $Z$. We assumed $M=\HH^{1}_*(\PP^2,\E)\neq0$, hence $\E\not\cong\OO_{\PP^2}\oplus\OO_{\PP^2}$. This implies that $Z\neq\emptyset$, and thus 
  $h^0(\PP^2,\mathcal{I})=0$.
  From the cohomology sequence
  $$0\to \HH^0(\PP^2,\OO_{\PP^2})\to \HH^0(\PP^2,\E)\to \HH^0(\PP^2,\mathcal{I})= 0,$$
  we obtain $h^0(\PP^2,\E)=h^0(\PP^2,\OO_{\PP^2})=1$.
\end{proof}

\subsection{Expected codimension of the non-Lefschetz locus of conics}
In this section, we want to talk about the codimension of the non-Lefschetz locus of conics $\C_M$.  The non-Lefschetz locus of conics $\C_M$ is defined a priori as a union of determinantal schemes. To compute the expected codimension, we need to show $\C_A$ is ``concentrated'' in one degree.

\begin{prop}
    The non-Lefschetz locus of $M=\HH^1_*(\PP^2,\E)$, for any rank $2$ vector bundle $\E$, a subscheme of $\PP^5$ coincides with the non-Lefschetz locus in the middle degree:
	$$\C_M =\C_{\lfloor \frac{d-3}{2}\rfloor-1,M}$$
 where $-d=c_1(\E)$ is the first Chern class of $\E$.
\end{prop}
\begin{proof}
    The proof is similar to what we did for Gorenstein algebras in Theorem \ref{'mid}.
    We know that $M$ has the WLP by \cite{FFP21}; as a consequence we have that:
    \begin{itemize}
        \item the Hilbert function of $M$ is unimodal;
        \item $M$ has no socle until degree $\lfloor \frac{d-3}{2}\rfloor$ and no new generators after that degree.
    \end{itemize}
    Moreover, by Serre Duality and the fact that $\check \E=\E(d),$ we have $$\HH^1\Big(\PP^2,\E\Big(\Big\lfloor  \frac{d-4}{2}\Big\rfloor-i\Big)\Big)\cong \HH^1\Big(\PP^2,\check\E\Big(-\Big\lfloor \frac{d-3}{2}\Big\rfloor+i-3\Big)\Big)\cong \HH^1\Big(\PP^2,\E\Big(\Big\lfloor  \frac{d-2}{2}\Big\rfloor+i\Big)\Big).$$
    Applying Proposition \ref{matrix} and Remark \ref{dual} and following the same proof as that of Theorem \ref{'mid} yields 
    $\C_M =\C_{\lfloor \frac{d-3}{2}\rfloor-1,M}.$
\end{proof}

This result assures us that the non-Lefschetz locus of conics has a determinantal structure, so we can compute the expected codimension.

\begin{prop}\label{expcod} Let $\E$ be a rank $2$ vector bundle over $\PP^2$. Then the non-Lefschetz locus of conics of $M=\HH^1_*(\PP^2,\E)$ has expected codimension 
    $$\exc\C_{M}=\begin{cases} 
    1& \text{ if } \E \text{ is unstable or has first Chern class odd;}\\
	2& \text{ if } \E \text{ is semistable but not stable;}\\
	3& \text{ if } \E \text{ is stable with $c_1(\E)$ even.}\\
\end{cases}$$	
\end{prop}
\begin{proof} 
Since $\C_M =\C_{\lfloor \frac{d-3}{2}\rfloor-1,M}$, it has a determinantal structure with expected codimension 
$$\exc \C_M= h_{\lfloor \frac{d-3}{2}\rfloor+1}-h_{\lfloor \frac{d-3}{2}\rfloor-1}+1$$
where $-d=c_1(\E)$ is the first Chern class of $\E$.

\emph{Case 1: odd Chern class.} If $d$ is odd, by Serre duality,  $h^1 (\PP^2,\E ( \frac{d-3}{2} -1 ) )= h^1 (\PP^2,\E ( \frac{d-3}{2} +1 ) )$, hence $\exc \C_M=1$.

\emph{Case 2: even Chern class.} Let us now consider $d$ even. Without loss of generality, we can assume $\E$ is normalized. Then $d=-c_1(\E)=0$, and 
$\exc \C_M= h_{-1}-h_{-3}+1.$

\emph{Case 2.1: $\E$ unstable.} If $\E$ is unstable with index of instability $k$, then  by \cite{FFP21} the multiplication map for a general line $\times \ell: [M]_{i-1}\to [M]_{i}$, is  bijective for $-(k+1) \leq i \leq k-1$. This implies 
 $h_{-k-2}=h_{-k-1}=\dots =h_{k-1}$, and in particular $h_{-3}=h_{-1}$ since $k>0$
when $\E$ unstable with $c_1(\E)=0$. So also in this case $\exc \C_M=1$.

\emph{Case 2.2: $\E$ semistable.} Let us now assume that $\E$ is semistable with $c_1(\E)=0$. Applying Serre Duality we have
\begin{align*}
	\exc  (\C_M ) =& \ h_{-1}-h_{-3}+1 = h^1(\PP^2,\E(-1))-h^1(\PP^2,\E(-3))+1\\
 =&\ h^1(\PP^2,\E(-2))-h^1(\PP^2,\E)+1\\
	=&\ -\chi(\E(-2))+h^0 (\PP^2,\E(-2))+h^2 (\PP^2,\E(-2))\\&+\chi (\E )-h^0 (\PP^2,\E)-h^2 (\PP^2,\E )+1\\
	=&\ \chi (\E)-\chi (\E(-2))-h^0 (\PP^2,\E )+1,
\end{align*}
since by semicontinuity $h^0 (\PP^2,\E(-2))=0$ as well as
\begin{align*}
  h^2 (\PP^2,\E(-2))=h^0 (\PP^2,\E(-1))=0\\
  h^2 (\PP^2,\E )=h^0 (\PP^2,\E(-3))=0.
\end{align*}
We can see that
$$\chi(\E(i))-\chi(\E(i-2))=\begin{cases}
		4i+2  & \text{ if }  c_1(\E)=0; \\
		4i &\text{ if }  c_1(\E)=-1,
	\end{cases}$$
 so $\chi (\E)-\chi (\E(-2))=2.$
If $\E$ is stable, then $h^0 (\PP^2,\E )=0$, while if $\E$ is semistable (and $M\neq0$) we show in the proof of Corollary \ref{jum3} that $h^0 (\PP^2,\E )=1$. Finally, we have
\begin{align*}
	\exc  (\C_M ) &= \chi (\E)-\chi (\E(-2))-h^0 (\PP^2,\E )+1= 3 -h^0 (\PP^2,\E )\\
 &=\begin{cases} 
	2& \text{ if } \E \text{ is semistable but not stable;}\\
	3& \text{ if } \E \text{ is stable.}
 \end{cases}
\end{align*}
\end{proof}

In the previous section, we saw that a general conic is a Lefschetz-conic, hence $\C_M\neq \PP^5$  and 
so $1\leq\cod (\C_M)\leq \exc (\C_M)$. This implies that if $\E$ is unstable or the first Chern class $c_1(\E)=-d$ is odd, the non-Lefschetz locus of conics of $M=\HH^1_*(\PP^2,\E)$ is always a hypersurface. In this case, the degree is given by $h_{\lfloor \frac{d-3}{2}\rfloor+1}$.
Note that by Theorem \ref{vit} and Corollary \ref{jum2} we already knew that for a semistable bundle $\E$ with first Chern class odd, the non-Lefschetz locus of conics is a hypersurface in $\PP^5$ of degree $c_2(\En)-1$.

It is left to study the non-Lefschetz locus of conics when $\E$ is semistable and $c_1(\E)=0$. By Theorem \ref{vit}  the set of jumping conics forms a hypersurface in $\PP^5$ of degree $c_2(\E)$, and by Corollary \ref{jum2} $\C_M$  is contained in this surface.

\begin{con}\label{conj}
For a general semistable vector bundle $\E$ with the first Chern class even, the non-Lefschetz locus has expected codimension  $$\cod\C_{M}=\begin{cases} 
	2& \text{ if } \E \text{ is semistable but not stable;}\\
	3& \text{ if } \E \text{  is stable.}\\
\end{cases}$$	
\end{con}
We will prove this conjecture in the case when $\E$ is the syzygy bundle of a complete intersection. We will also show that the hypothesis of generality is necessary, using examples of monomial complete intersections.

\section{General complete intersections of height $3$}

In this section, we focus on Artinian complete intersections.
Let $A=\frac{\kk[x_1,x_2,x_3]}{(f_1,f_2,f_3)}$ be a complete intersection\footnote{We say that a complete intersection $A=\frac{\kk[x_1,x_2,x_3]}{(f_1,f_2,f_3)}$ has \emph{type} $(d_1,d_2,d_3)$, if $\deg f_i=d_i$ for $i=1,2,3$. We will always assume that $2\leq d_1\leq d_2\leq d_3$. In this case, $A$ is an algebra of height (or codimension) $3$. Throughout the paper, we will use the term \emph{height} when referring to an algebra, reserving the term \emph{codimension} for geometric settings.} of type $(d_1,d_2,d_3)$, and let $\E$ be its first syzygy bundle.
Our goal is to prove Conjecture \ref{conj} for $\E$. Recall that $A\cong \HH_*^1(\PP^2, \E)$. Moreover, $\E$ has first Chern class odd if and only if the socle degree $e$ is even, and
\begin{itemize}
    \item $\E$  is stable if $d_3< d_1+d_2$;
    \item $\E$ is semistable if $d_3\leq d_1+d_2$;
    \item $\E$ is unstable if $d_3> d_1+d_2.$
\end{itemize}
 
\begin{teo}\label{main CI}
	If $A=R/(f_1,f_2,f_3)$ is a general complete intersection of type $(d_1,d_2,d_3)$, then the non-Lefschetz locus of conics has the expected codimension in $\PP^5$:
	$$\cod \C_{A}=\begin{cases} 1& \text{ if } e \text{ is even;}\\
		1& \text{ if } d_3> d_1+d_2;\\
		2& \text{ if } d_3=d_1+d_2;\\
		3& \text{ if } e \text{ is odd and }d_3\leq d_1+d_2-2.\\
	\end{cases}$$
\end{teo}

Note that when either the vector bundle $\E$ is unstable or its first Chern class is odd, we know that the expected codimension is achieved and $\C_A$ is a hypersurface. This happens when either the socle degree $e$ is even or $d_3> d_1+d_2$. Therefore, in this case, the non-Lefschetz locus of conics is a hypersurface of degree $h_{\lfloor\frac{e}{2}\rfloor-1}$.

By this reasoning, we can restrict to the case when the socle degree $e$ is odd and $d_3\leq d_1+d_2$. Equivalently, the first syzygy bundle $\E$ is semistable with the first Chern class even.

Let $(1,3,h_2,\dots,h_{e-2},3,1)$ be the Hilbert function of a complete intersection of type $(d_1,d_2,d_3)$. We will construct a Gorenstein algebra $R/J$ with the same Hilbert function $$(1,3,h_2,\dots,h_{e-2},3,1)$$ such that the non-Lefschetz locus of conics has expected codimension.
We know that when the Hilbert function is fixed, the height $3$ Gorenstein algebras with such Hilbert function lie in a flat family \cite{Diesel}. Then, by semicontinuity, we can conclude that the general Gorenstein algebra with that Hilbert function has non-Lefschetz locus of expected codimension.
Moreover, from \cite{Diesel} we know the Gorenstein algebras with the minimum number of generators and Hilbert function $(1,3,h_2,\dots,h_{e-2},3,1)$ form a Zariski dense set in the family of Gorenstein algebras with this Hilbert function. 
In our case, we know that  $(1,3,h_2,\dots,h_{e-2},3,1)$ is the Hilbert function of a complete intersection, so in this family of Gorenstein algebras the ones with the minimum number of generators must be the complete intersections. Since this set is dense, we can conclude by semicontinuity that the general complete intersection has the non-Lefschetz locus of conics of expected codimension, assuming that we have constructed a Gorenstein algebra with the same Hilbert function and non-Lefschetz locus of conics of expected codimension.

\subsection{Proof of Theorem \for{toc}{\ref*{main CI}}\except{toc}{\ref{main CI}}}

Before proceeding with the proof of Theorem \ref{main CI}, we need to recall some definitions and results about Gorenstein algebras of height $3$. Recall that a sequence $(1,3,h_2,\dots,h_{e})$ is a Hilbert function for a height $3$ Gorenstein algebra if and only if it is a \emph{SI-sequence}.
\begin{df}
	A height $3$ \emph{SI-sequence} is a sequence of positive integers $(1,3,h_2,\dots,h_{e})$ such that \begin{itemize}
		\item it is symmetric;
		\item the first difference $(1,2,h_2-h_1,\dots,h_{\lfloor \frac{e}{2}\rfloor}-h_{\lfloor \frac{e}{2}\rfloor-1})$ satisfies Macaulay's growth condition.
	\end{itemize}
\end{df}
The second condition is equivalent to the statement that  $$(1,3,h_2,\dots,h_{\lfloor \frac{e}{2}\rfloor}, h_{\lfloor \frac{e}{2}\rfloor}, \dots, h_{\lfloor \frac{e}{2}\rfloor}, \dots )$$ is the Hilbert function of zero-dimensional subscheme $Z$ of $\PP^2$.

For a fixed SI-sequence, the family of Gorenstein algebras $R/I$ having that sequence as a Hilbert function is an irreducible family by \cite{Diesel}.

\begin{df}
    Let $(1,3,h_2,\dots,h_{e})$ be the Hilbert function of a Gorenstein algebra $A$. The \emph{g-vector} of $A$ is the positive part of the first difference: $$(1,2,g_2,\dots,g_{\lfloor \frac{e}{2}\rfloor})=(1,2,h_2-h_1,\dots,h_{\lfloor \frac{e}{2}\rfloor}-h_{\lfloor \frac{e}{2}\rfloor-1}).$$
    We say that the g-vector $(1,2,g_2,\dots,g_{\lfloor \frac{e}{2}\rfloor})$ is of \emph{decreasing type} if it begins with $(1,2,3,\dots)$, then is possibly constant, then is strictly decreasing.
\end{df}

\begin{df}
	We say that the Gorenstein algebra $A$ with Hilbert function $(1,3,h_2,\dots,h_e)$, \emph{comes from points} if it is a quotient of $R/I_Z$  where $I_Z$ is the ideal associated to a  reduced zero-dimensional scheme $Z$ with Hilbert function $$(1,3,h_2,\dots,h_{\lfloor \frac{e}{2}\rfloor},h_{\lfloor \frac{e}{2}\rfloor} \dots).$$
\end{df}

For any SI-sequence, there is always a subfamily of Gorenstein algebras with that sequence as a Hilbert function which comes from points by \cite{BOIJ}.
\begin{rmk}\label{points}
	A general complete intersection $R/(f_1,f_2,f_3)$ comes from points if and only if $d_3\geq d_1+d_2-1$. In this case, $[R/(f_1,f_2,f_3)]_j= [R/(f_1,f_2)]_j$ for $j<\frac{e-1}{2}$ and the Hilbert function of $R/(f_1,f_2)$ stabilizes at $d_1+d_2-2\leq \frac{e-1}{2}$. 
\end{rmk} 

The last definition that we need to recall is a strong form of general position for sets of points $Z$ in $\PP^2$:
\begin{df}
    Let $Z$ be a set of points in $\PP^2$. We say that $Z$ has the \emph{Uniform Position Property (UPP)} if, for any $n\leq \deg Z$, all subsets of $n$ points have the same Hilbert function. 
\end{df}
As a consequence, if a set of points $Z$ in $\PP^2$ has the UPP, and $Y$ is a subset of $n$ points of $Z$, we have that the Hilbert function $h_i(Y)=\dim[R/I_Y]_i$ of $Y$ must be the truncation of the Hilbert function of $Z$: $$h_i(Y)=\min\{h_i(Z), n\},$$ where $h_i(Z)=\dim[R/I_Z]_i$ is the Hilbert function of $Z$.

Now we can proceed with the proof of Theorem \ref{main CI}: we want to show that the non-Lefschetz locus of conics $\C_A$ of a general complete intersection $A=R/(f_1,f_2,f_3)$ of type $(d_1,d_2,d_3)$ has expected codimension. The only case left to prove is when the socle degree $e$ is odd and $d_3\leq d_1+d_2$.

\begin{proof}[Proof of Theorem \ref{main CI}] 
Let $A=R/(f_1,f_2,f_3)$ be a complete intersection of type $(d_1,d_2,d_3)$, where $e$ is odd and $d_3\leq d_1+d_2$. Let $(1,3,h_2,\dots,h_{e-2},3,1)$ be the Hilbert function of $A$.
By Proposition \ref{expcod},  the expected codimension of $\C_A$ is
\begin{itemize}
    \item $2$ when $d_3=d_1+d_2$ (this corresponds to the case when the syzygy bundle $\E$ is semistable but not stable);
    \item $3$ otherwise (when the syzygy bundle $\E$ is semistable with $c_1(\E)$ even).
\end{itemize}

First, we show that the g-vector is always of decreasing type. Assume by contradiction that
$$(1,2,g_2,\dots,g_{\lfloor \frac{e}{2}\rfloor})=(1,2,h_2-h_1,\dots,h_{\lfloor \frac{e}{2}\rfloor}-h_{\lfloor \frac{e}{2}\rfloor-1})$$
is not of decreasing type. Then there exists $i<\frac{e-1}{2}$ such that $g_{i-2}>g_{i-1}=g_i$. Since $g_{i-2}>g_{i-1}$, we have that $d_1, d_2 \leq i$. Now $g_{i-1}=g_i$ so by \cite[Theorem 3.1]{RZ} all the generators of degree $\leq i$ have a common factor of degree $g_i$. So $f_1$ and $f_2$ have a common factor, but this is not possible since $(f_1,f_2,f_3)$ is a complete intersection.

We will now construct a Gorenstein algebra $R/J$ with the same Hilbert function of a complete intersection of type $(d_1,d_2,d_3)$ such that the non-Lefschetz locus of conics has expected codimension. Let $R/J$ be an Artinian Gorenstein algebra with  Hilbert function $$(1,3,h_2,\dots,h_{e-2},3,1)$$  that comes from points. So $R/J$ is obtained as a quotient of $R/I_Z$ where $Z$ is a reduced zero-dimensional scheme in $\PP^2$.
Therefore $[R/I_Z]_i=[R/J]_i$ for every $i\leq \frac{e-1}{2}$; moreover, since $h_\frac{e-1}{2}=h_ \frac{e+1}{2}$ and the Hilbert function of $I_Z$ stabilizes at $ \frac{e-1}{2}$, we also have that 
$[R/I_Z]_\frac{e+1}{2}=[R/J]_\frac{e+1}{2}$.
Since we showed that the positive first difference of  $(1,3,h_2,\dots,h_{e-2},3,1)$ is of decreasing type, we can assume $Z$  satisfies the UPP by \cite{MR}.

By Theorem \ref{'mid}, $C_{R/J}$ is ``concentrated'' in the middle degree, i.e. $\C_{R/J}= \C_{{R/J},\frac{e-1}{2}-1}.$ Then under the identification

 \[ \begin{tikzcd}
\left[R/J\right]_{\frac{e-1}{2}-1} \arrow{r}{\times C}\arrow[equal]{d}
& \left[R/J\right]_{\frac{e+1}{2}} \arrow[equal]{d}\\
\left[R/I_Z\right]_{\frac{e-1}{2}-1} \arrow{r}{\times C}&\left[R/I_Z\right]_{\frac{e+1}{2}}
\end{tikzcd}
\]
$C$ is a non-Lefschetz conic for $R/J$ if and only if $\times C: [R/I_Z]_{\frac{e-1}{2}-1}\to [R/I_Z]_{\frac{e+1}{2}}$ is not injective.
If $C$ does not pass through any of the points of $Z$, then $C$ is a non-zero divisor and so the multiplication map is injective. Let us assume that $C$ passes through at least one point of $Z$ and let $I_Y=(I_Z:C)$ be the ideal associated to the set of points $Y$ of $Z$ that $C$ does not pass through. Then the map $\times C: [R/I_Z]_{\frac{e-1}{2}-1}\to [R/I_Z]_{\frac{e+1}{2}}$ is injective if and only if $[I_Y/I_Z]_{\frac{e-1}{2}-1}=0$.
We want to check whether the Hilbert function of $I_Y$ and $I_Z$ are equal in degree $\frac{e-1}{2}-1$.

Since we assumed $Z$ has the UPP, the Hilbert function of  $I_Y$ depends only on the number of points of $I_Y$, and so it must be the truncated Hilbert function of $Z$; in particular
$$\dim[R/I_Y]_{\frac{e-1}{2}-1}=\min\{\dim[R/I_Z]_{\frac{e-1}{2}-1}, n\}.$$

\emph{Case 1: $d_3= d_2+d_1$.} In this case $\exc \C_A =2$, hence $h_{\frac{e-1}{2}-1}-h_{\frac{e+1}{2}}=1$.
Since  $$\dim([R/I_Z]_{\frac{e-1}{2}})-\dim([R/I_Z]_{\frac{e-1}{2}-1})=h_{\frac{e-1}{2}-1}-h_{\frac{e-1}{2}}=h_{\frac{e-1}{2}-1}-h_{\frac{e+1}{2}}=1,$$
if $C$ meets just one of the points of $Z$ then $[I_Y/I_Z]_{{\frac{e-1}{2}-1}}=0$. If $C$ passes through exactly $2$ points of $Z$, then $\dim [I_Y]_{{\frac{e-1}{2}-1}}=\dim [I_Z]_{{\frac{e-1}{2}-1}}-1$ and so the map $\times C$ is not injective.
We can conclude that $C\in \C_{R/J}$ if and only if $C$ passes through at least $2$ points of $Z$.
Then $\cod \C_{R/J}=2$ as we wanted.

\emph{Case 2: $d_3\leq d_2+d_1-2$.}   Here $\exc \C_A =3$, hence $h_{\frac{e-1}{2}-1}-h_{\frac{e+1}{2}}=2$.
In this case $$\dim([R/I_Z]_{\frac{e-1}{2}})-\dim([R/I_Z]_{\frac{e-1}{2}-1})=h_{\frac{e-1}{2}-1}-h_{\frac{e-1}{2}}=h_{\frac{e-1}{2}-1}-h_{\frac{e+1}{2}}=2,$$ and so
  $\dim [I_Y]_{\frac{e-1}{2}-1}\neq\dim [I_Z]_{\frac{e-1}{2}-1}$ if $C$ passes through at least $3$ points of $Z$. 
 We can conclude that $\cod \C_{R/J}=3$ as expected. 
 
 By semicontinuity, the non-Lefschetz locus of a general Gorenstein ideal with Hilbert function  $(1,3,h_2,\dots,h_{e-2},3,1)$ has codimension $3$. In particular, for a general complete intersection $I$ of type $(d_1,d_2,d_3)$, the non-Lefschetz locus has expected codimension.  
\end{proof}

\subsection{Note about the case \for{toc}{$d_3\leq d_1+d_2-4$}\except{toc}{\boldmath$d_3\leq d_1+d_2-4$} and odd socle degree}

In Remark \ref{points} we notice that for $d_3\geq d_1+d_2-1$, the Artinian complete intersection $A=R/(f_1,f_2,f_3)$ comes from points. Moreover, if $Z$ is the set of points defined as a zero-dimensional scheme by the ideal $(f_1,f_2)$, we have that 
$[A]_{j}=[R/(f_1,f_2)]_{j}$ for any $j<d_3$. Since $d_3\geq d_1+d_2-1$,
$$\floor*{ \frac{e}{2}}+1\geq \frac{d_3+d_3-2}{2}+1=d_3-1.$$

By Theorem \ref{'mid}, $\C_{A}= \C_{{A},\lfloor \frac{e}{2}\rfloor-1},$ so $C$ is a non-Lefschetz conic if and only if
the map
 \[ \begin{tikzcd}
\left[R/(f_1,f_2,f_3)\right]_{\lfloor \frac{e}{2}\rfloor-1} \arrow{r}{\times C}\arrow[equal]{d}
& \left[R/(f_1,f_2,f_3)\right]_{\lfloor \frac{e}{2}\rfloor+1} \arrow[equal]{d}\\
\left[R/(f_1,f_2)\right]_{\lfloor \frac{e}{2}\rfloor-1} \arrow{r}{\times C}&\left[R/(f_1,f_2)\right]_{\lfloor \frac{e}{2}\rfloor+1}
\end{tikzcd}
\]
 is not injective. Then a conic $C$ that is in the non-Lefschetz locus of conics of $A=R/(f_1,f_2,f_3)$ needs necessarily to vanish at least at one of the points of $Z$ (this condition is not necessarily sufficient, as we can see in the case $d_3=d_1+d_2$). 
 
 This is not true in general. In fact, the subset of the non-Lefschetz locus of conics that do not vanish at any point of $Z$ for a general complete intersection $A=R/(f_1,f_2,f_3)$ with $d_3\leq d_1+d_2-4$ and $e$ odd has codimension  $3$ in $\PP^5$. This agrees with the codimension of the entire scheme $\C_A$ by Theorem \ref{main CI}. 

\begin{prop}
	Let $A=R/(f_1,f_2,f_3)$ be a general complete intersection with $d_3\leq d_1+d_2-4$ and $e$ odd.
	Assume $Z$ is the set of points defined as a zero-dimensional scheme by the ideal $(f_1,f_2)$.
	The set of conics $C$  in the non-Lefschetz locus of conics $\C_{A}$ that do not vanish at any of the points of $Z$ has codimension $3$ in $\PP^5$.
\end{prop}
\begin{proof}
	Let $A=R/(f_1,f_2,f_3)$ be a general complete intersection and let $Z$ be the set of points defined as a zero-dimensional scheme by the ideal $(f_1,f_2)$. We consider just the set of conics $C$ that do not pass through any of the points of $Z$, or equivalently, the conics such that $(C,f_1,f_2)$ is a complete intersection. As before $(1,3,h_2,\dots,h_{e-2},3,1)$ is the Hilbert function of $R/I$. Here, we are also assuming that the socle degree $e$ is odd.

	To check if a conic $C$ is a Lefschetz conic it is enough to verify whether the multiplication map $$\times C: [A]_{\frac{e-1}{2}}\to [A]_{\frac{e+1}{2}+1}$$ is surjective, or equivalently if   $[R/(C,f_1,f_2,f_3)]_{\frac{e+1}{2}+1}$ is zero.
	A general conic $C$ is a Lefschetz conic, so $[R/(C,f_1,f_2,f_3)]_{\frac{e+1}{2}+1}=0$. In this case $(C,f_1,f_2,f_3)$ is an almost complete intersection of type $(2,d_1,d_2,d_3)$ with Hilbert function $$(1,3,h_2-1,h_3-h_1,\dots,h_{\frac{e-1}{2}}-h_{\frac{e-1}{2}-2},h_{\frac{e+1}{2}}-h_{\frac{e-1}{2}-1}).$$
	We want to look now at the conics $C\in \C_{R/I}$, so that $[R/(C,f_1,f_2,f_3)]_{\frac{e+1}{2}+1}\neq0$.
	Since we are just considering conics that don't vanish at any point of $Z$, the ideal $(C,f_1,f_2)$ is a complete intersection. To compute the codimension it is enough to consider $C$ such that $(C,f_1,f_2,f_3)$ is an almost complete intersection whose Hilbert function differs from the general case in the least possible way. Note that if $$\dim\left[\frac{R}{(C,f_1,f_2,f_3)}\right]_{\frac{e+1}{2}+1}=1,$$ by duality we also have $$\dim \ker \left(\times C: \left[\frac{R}{(C,f_1,f_2,f_3)}\right]_{\frac{e-1}{2}-1}\to \left[\frac{R}{(C,f_1,f_2,f_3)}\right]_{\frac{e+1}{2}}\right)=1,$$ and $\dim[R/(C,f_1,f_2,f_3)]_{\frac{e+1}{2}}=h_{\frac{e+1}{2}}-h_{\frac{e-1}{2}-1}+1$. Then we can assume that the Hilbert function of $R/(C,f_1,f_2,f_3)$ is  
	$$(1,3,h_2-1,h_3-h_1,\dots,h_{\frac{e-1}{2}}-h_{\frac{e-1}{2}-2},h_{\frac{e+1}{2}}-h_{\frac{e-1}{2}-1}+1,1).$$
	We want to compute the codimension of the almost complete intersection of type $(2,d_1,d_2,d_3)$ in $R$ with Hilbert function  $$(1,3,h_2-1,h_3-h_1,\dots,h_{\frac{e+1}{2}}-h_{\frac{e-1}{2}-1}+1,1)$$
	in the space of all almost complete intersections of type $(2,d_1,d_2,d_3)$.
	A general almost complete intersection $J$  of type $(2,d_1,d_2,d_3)$ has Hilbert function $$(1,3,h_2-1,h_3-h_1,\dots,h_{\frac{e+1}{2}}-h_{\frac{e-1}{2}-1}).$$
	We can link $J$ by a complete intersection $K$ of
	type $(2,d_1, d_2)$ to a Gorenstein ideal $G$ with socle degree $s = d_1 + d_2  - d_3 - 1$ and
	h-vector $(1, f_1, \dots , f_s)$. In a similar way, we link a complete intersection $J'$  with Hilbert function  $(1,3,h_2-1,h_3-h_1,\dots,h_{\frac{e+1}{2}}-h_{\frac{e-1}{2}-1}+1,1)$ by a complete intersection $K$ of type $(2,d_1, d_2)$ to a Gorenstein ideal $G'$ with socle degree $s$ and h-vector $(1, f'_1, \dots, f'_s )$ that differs from the one of $G$ only in the middle degrees. In fact, since the Hilbert function of $J$ and $J'$ differ by one in degree $\frac{e}{2}$ and $\frac{e}{2}+1$ and are equal otherwise, we can relate the Hilbert function of $G$ and $G'$ as follows:
	\[\begin{cases}
		f'_i=f_i-1 &\text{ if } i=\frac{s-1}{2} \text{ or } i=\frac{s+1}{2};\\
		f'_i=f_i &\text{ otherwise.}
	\end{cases}\]

	We follow a similar method to the proof of \cite[Theorem 5.6]{main}. Using the same notation we denote by
	\begin{enumerate}
		\item $D_1$, resp. $D_1'$, the dimension of the family of Gorenstein Artinian ideals $G$, resp. $G'$, with Hilbert function $(1, f_1, \dots, f_s )$, $(1, f'_1, \dots, f'_s )$ respectively;
		\item $D_2$, resp. $D_2'$, the dimension of complete intersections $K$ of type $(2,d_1, d_2)$ contained in $G$, $G'$ respectively; 
		\item $D_3$, resp. $D_3'$, the dimension of complete intersections $K$ of type $(2, d_1, d_2)$ contained in $J$, $J'$ respectively.
	\end{enumerate}
	Then the codimension of the almost complete intersection that we are looking at is exactly   
	$ D_1+D_2-D_3 -( D'_1+D'_2-D'_3)$.
	The reason why we subtract $D_3$ ($D'_3$ respectively) is to remove over-counting, since the same ideal $J$ ($J'$)
	can be reached from many different ideals $G$ ($G'$) using different complete intersections in $J$ ($J'$).

	 Since the Hilbert function $J$ and $J'$ are different only in degrees $\frac{e+1}{2}$ and $\frac{e-1}{2}+1$, $D_3\neq D_3'$ only in the case when one of the degrees $2$, $d_1$, $d_2$ is equal to $\frac{e+1}{2}$ or $\frac{e+1}{2}+1$. But this is not possible because 
	$d_3\leq d_1+d_2-4$, therefore $2\leq d_1\leq d_2 \leq d_3 < \frac{e+1}{2}$. Hence, $D_3=D_3'$. 
	
	Similarly, $D_2\neq D'_2$ only when one of the degrees $2$, $d_1$, $d_2$ is equal to $\frac{s\pm1}{2}$, since the Hilbert functions of $G$ and $G'$ differ only in the middle degrees.
	This can happen if and only if $d_3= d_1+d_2-4$ or $d_3= d_1+d_2-6$. In both cases $D_2-D'_2=-1$.
	
	Finally, to compute $D_1-D'_1$,  we can apply \cite[Lemma 5.5]{main}: $G$ and $G'$ are Gorenstein algebras of height $3$ with odd socle degree and Hilbert functions that differ by one only in the middle degree. Thus 
	$D_1-D'_1=f_{\frac{s-1}{2}+1}-2f_{\frac{s-1}{2}+3} +f_{\frac{s-1}{2}+4}+1$.
	Recall that we obtained the Gorenstein Artinian ideal $G$ from $J$ linking by a complete intersection $K$ of type $(2,f_1,f_2)$.
	Let $(1,3,\tilde{h}_2,...,\tilde{h}_{d_1+d_2-1},3,1)$ be the Hilbert function of $K$, with socle degree $d_1+d_2-1$.
	Using the property of linkage and symmetry we have that 
	$$f_{\frac{s-1}{2}+1+j}=\tilde{h}_{\frac{d_1+d_2+d_3}{2}+j}=\tilde{h}_{\frac{d_1+d_2-d_3-2}{2}-j}=\binom{\frac{d_1+d_2-d_3-2}{2}-j+2}{2}-\binom{\frac{d_1+d_2-d_3-2}{2}-j}{2},$$
	where the last equality uses the fact that $\frac{d_1+d_2-d_3-2}{2}\leq d_1$ and in our case $j=0,2,3$.
	After some numerical computation, we obtain that 
	$$D_1-D'_1=f_{\frac{s-1}{2}+1}-2f_{\frac{s-1}{2}+3} +{\frac{s-1}{2}+4}+1=\begin{cases}
		4 &\text{ if } d_3=d_1+d_2-4 \\
  &\ \ \ \text{or } d_3=d_1+d_2-6;\\
		3 &\text{ if } d_3\leq d_1+d_2-8.
	\end{cases}$$
	We can conclude that the codimension of the conics $C$  in the non-Lefschetz locus of conics $\C_{A}$ that do not vanish at any of the points of $Z$ in $\PP^5$ is
	\begin{align*}
		&(D_1-D'_1)+(D_2-D'_2)-(D_3-D'_3)=\\ &\begin{cases}
			4 &\text{if } d_3=d_1+d_2-4\\
   &\ \ \ \text{or } d_3=d_1+d_2-6;\\
			3 &\text{if } d_3\leq d_1+d_2-8; 
		\end{cases} + 
	\begin{cases}
		-1 &\text{if } d_3=d_1+d_2-4\\
  &\ \ \ \text{or } d_3=d_1+d_2-6;\\
		\ 0 &\text{if } d_3\leq d_1+d_2-8.
	\end{cases}
	\end{align*}
	In all cases, this is $3,$ as required. 
\end{proof}

\section{Examples with Monomial Complete Intersections}
In this section, we show that the hypothesis of generality in Theorem \ref{main CI} is necessary by constructing examples of monomial complete intersections where the non-Lefschetz locus of conics does not have the expected codimension.
We already know that for every Artinian monomial complete intersection $A=\kk[x_1,x_2,x_3]/(x_1^{d_1}, x_2^{d_2}, x_3^{d_3})$  the non-Lefschetz locus of conics is a hypersurface in $\PP^5$ if $d_3>d_1+d_2+1$ or if the socle degree $e=d_1+d_2+d_3-3$ is odd. Therefore, in this section we focus on the case when the socle degree $e$ is even and $d_3\leq d_1+d_2$.

First, if $d_3= d_1+d_2$, then $\C_A$ has expected codimension $2$, but this codimension is never achieved by a monomial complete intersection:
\begin{prop}
    The non-Lefschetz locus of conics of an Artinian monomial complete intersection $A=\kk[x_1,x_2,x_3]/(x_1^{d_1}, x_2^{d_2}, x_3^{d_3})$ with $d_3=d_1+d_2$ has codimension $1$ in $\PP^5$.
\end{prop}
\begin{proof}
   As in \emph{Case 2.2} of Theorem \ref{main CI}, $C$ is a Lefschetz conic if and only if the multiplication map  $\times C: [A]_{\frac{e-1}{2}-1}\to [A]_{\frac{e+1}{2}}$ is injective. 
Since $\frac{e+1}{2}=d_1+d_2-1=d_3-1<d_3$, it is equivalent to check when the map  $\times C: [R/(x_1^{d_1},x_2^{d_2})]_{d_1+d_2-3}\to [R/(x_1^{d_1},x_2^{d_2})]_{d_1+d_2-1}$ is injective.  We can see that $\times C$ is not injective if and only if $a_6=0$, where $$C=a_1x_1^2+ a_2x_1x_2+ a_3x_1x_3+ a_4x_2^2+ a_5x_2x_3 + a_6x_3^2.$$
Clearly if $a_6\neq 0$, then $C$ is not a zero-divisor in $R/(x_1^{d_1},x_2^{d_2})$ and $\times C$ must be injective. If $a_6=0$,  we can check that if $a_3\neq 0$ then
$$x_1^{d_1-1}x_2^{d_2-2}-\frac{a_5}{a_3}x_1^{d_1-2}x_2^{d_2-1}\stackrel{\times c}\mapsto 0,$$
while if $a_3=0$ we have
$$x_1^{d_1-2}x_2^{d_2-1}\stackrel{\times c}\mapsto 0.$$
So $C=a_1x_1^2+ \dots + a_6x_3^2\in \C_A$ if and only if $a_6=0$, and $\cod \C_A=1$.
\end{proof}
This proposition gives us that if $d_1+d_2=d_3$, then the non-Lefschetz locus of conics is defined as a set by the ideal $(a_6)$. However, as a subscheme of $\PP^5$, it may not be the case that $\C_A$ is reduced or unmixed, as we can see in the following example.

\begin{ex}\label{ex semi}
	Using \texttt{Macaulay2} \cite{M2} we obtain that the non-Lefschetz locus of conics of the monomial complete intersection  
 $$A=\kk[x_1,x_2,x_3]/(x_1^2, x_2^2, x_3^4)$$
 is defined by the ideal 
 $$I(C_A)=\left(a_{6}^{3},a_{5}a_{6}^{2},-a_{3}a_{6}^{2},-2a_{3}a_{5}a_{6}+a_{2}a_{6}^{2}\right).$$
 While the expected codimension is $2$ in this case, in agreement with the previous proposition,
  $\cod \C_A=1$ and $\sqrt{I(\C_A)}=(a_6)$.
 This ideal is saturated, but is not radical and not unmixed:  the primary decomposition 
 $$\left\{\left( a_{6}\right),\left(a_{5},a_{6}^{2}\right),\left(a_{6}^{2},a_{3}a_{6},a_{3}^{2}\right),\left(a_{5}^{2},a_{3}^{2},a_{6}^{3},a_{5}a_{6}^{2},a_{3}a_{6}^{2},2a_{3}a_{5}a_{6}-a_{2}a_{6}^{2}\right)\right\}$$
 has components of respective codimensions $1$, $2$, $2$ and $3$.
\end{ex}

Let us consider  now monomial complete intersections $$\frac{\kk[x_1,x_2,x_3]}{(x_1^{d_1}, x_2^{d_2}, x_3^{d_3})}$$
with $d_3> d_1+d_2$ and even socle degree $e$. In this case, the expected codimension is $3$, but examples show that all possible codimensions are achieved.  
\begin{ex}
	For the monomial complete intersection $A=R/(x_1^2,x_2^2,x_3^2)$, the non-Lefschetz locus of conics is given by the ideal generated by the coefficients of the square-free monomial of $C$:
 $$I(\C_A)=\left(a_5, a_3, a_2\right),$$
 so it has the expected codimension $\cod\mathcal{C}_A=3$.
 \end{ex}

\begin{ex}\label{ex sta1}
 For $A=\kk[x_1,x_2,x_3]/(x_1^3, x_2^3, x_3^4)$,  we have that $\cod \C_A=2$.
 It is interesting to note that $I(\C_A)$ is neither saturated nor unmixed: we can show using \texttt{Macaulay2} \cite{M2} that the primary components have codimension respectively $2,3,3,3,3,3,3,3,$ and $6$. The last component is Artinian, and does not correspond to any geometric component. Unlike Example \ref{ex semi}, the radical is also not unmixed, with the first component having codimension $2$ and the rest having codimension $3$:
 
\tiny{\begin{align*}
   \sqrt{I(\C_A)}=&(a_{6},\,a_{3}^{2}a_{4}-a_{2}a_{3}a_{5}+a_{1}a_{5}^{2})\cap(a_{5},\,a_{4},\,a_{2})\cap (\,a_{5},\,a_{3},\,a_{2}^{2}+2\,a_{1}a_{4})\cap(a_{3},\,a_{2},\,a_{1})\cap\\
   &(3\,a_{3}a_{5}^{2}-2\,a_{3}a_{4}a_{6}-2\,a_{2}a_{5}a_{6},\,12\,a_{1}a_{5}^{2}+a_{2}^{2}a_{6}-10\,a_{1}a_{4}a_{6},\,2\,a_{3}a_{4}a_{5}-a_{2}a_{5}^{2}-6\,a_{2}a_{4}a_{6},
   3\,a_{3}^{2}a_{5}-2\,a_{2}a_{3}a_{6}-2\,a_{1}a_{5}a_{6},\\
   &\,12\,a_{2}a_{3}a_{5}-7\,a_{2}^{2}a_{6}-2\,a_{1}a_{4}a_{6},\,12\,a_{3}^{2}a_{4}+a_{2}^{2}a_{6}-10\,a_{1}a_{4}a_{6},\,8\,a_{2}a_{3}a_{4}+a_{2}^{2}a_{5}-2\,a_{1}a_{4}a_{5},\,a_{2}a_{3}^{2}-2\,a_{1}a_{3}a_{5}+6\,a_{1}a_{2}a_{6},\\
   &a_{2}^{2}a_{3}-2\,a_{1}a_{3}a_{4}+8\,a_{1}a_{2}a_{5},\,a_{5}^{4}+4\,a_{4}a_{5}^{2}a_{6}-4\,a_{4}^{2}a_{6}^{2},\,a_{3}^{4}+4\,a_{1}a_{3}^{2}a_{6}-4\,a_{1}^{2}a_{6}^{2},\,a_{2}^{4}-68\,a_{1}a_{2}^{2}a_{4}+4\,a_{1}^{2}a_{4}^{2}).
\end{align*}}

 \end{ex}
 \begin{ex}\label{ex sta2}
 Finally, if $A=R/(x_1^4,x_2^4,x_3^6)$, then the non-Lefschetz locus of conics has $\cod\mathcal{C}_A=1$.
 Also in this case, the ideal is neither saturated nor unmixed. Note that even if $\cod\mathcal{C}_A=1$ and it is contained in the hypersurface parametrizing the jumping conics, these do not coincide.
 Using \texttt{Macaulay2} \cite{M2}  we can show that $ \sqrt{I(\C_A)}$ has one primary component of codimension $1$, the ideal $(a_6)$, and all other components have codimension $3$.
\end{ex}

\section{General Gorenstein Algebras}
The construction of a Gorenstein algebra with the non-Lefschetz locus of expected codimension in the proof of Theorem \ref{main CI} suggests that it may be possible to generalize the result to general Gorenstein algebras. 
 We saw that the Gorenstein algebras $R/I$ with fixed  Hilbert function form an irreducible family by \cite{Diesel}, so by ``general Gorenstein algebra'' we refer to a general element in this family. 

In this section we want to compute the codimension for the non-Lefschetz locus of conics $\C_A$ of a general Gorenstein algebra, fixing the Hilbert function. We will show that the condition on the g-vector is necessary to get the expected codimension if $h_{\lfloor \frac{e}{2}\rfloor-1}\neq h_{\lfloor \frac{e}{2}\rfloor+1}$.

By Proposition \ref{'mid}, we know $C_A=\C_{A,\lfloor\frac{e}{2}\rfloor-1}$, so the non-Lefschetz locus of conics has expected codimension $$\exc \C_{A} = \min \{ h_{\lfloor \frac{e}{2}\rfloor+1}-h_{\lfloor \frac{e}{2}\rfloor-1}+1, 6\}.$$
By \cite[Proposition 3.2]{AAISY},
we know that a general Gorenstein algebra has the Strong Lefschetz Property; in particular, there exists a linear form $\ell$ such that $\times \ell^2: [A]_i\to [A]_{i+2}$ has maximum rank for each $i$, and so $\ell^2\notin\C_A$. As a consequence, we know that for a general Gorenstein algebra $A$ we have $\cod \C_A\geq 1$.
If follows that, if $h_{\lfloor \frac{e}{2}\rfloor-1}=h_{\lfloor \frac{e}{2}\rfloor+1},$ then the non-Lefschetz locus must have expected codimension.
\begin{rmk}\label{even}
    the above observations imply that if $A$ is a general Gorenstein algebra with Hilbert function $(1,3,h_2,\dots,h_{e})$ and even socle degree $e$, then the non-Lefschetz locus always has the expected codimension: it is a hypersurface in $\PP^5$ of degree $h_{\frac{e}{2}-1}$.
\end{rmk}

\begin{prop}\label{genGo}
	Let $(1,3,h_2,\dots,h_{e})$ be an SI-sequence such that the positive part of its first difference $(1,2,g_2,\dots,g_{\lfloor \frac{e}{2}\rfloor})$ is of decreasing type.
	Then for a general Gorenstein algebra $A$ with Hilbert function $(1,3,h_2,\dots,h_{e}),$ the non-Lefschetz locus of conics has expected codimension in $\PP^5$
	$$\cod \C_{A} = \min \{ h_{\lfloor \frac{e}{2}\rfloor+1}-h_{\lfloor \frac{e}{2}\rfloor-1}+1, 6\}.$$
\end{prop}
\begin{proof}
    If $e$ is even, then $\C_A$ has expected codimension by Remark \ref{even}, so we can assume $e$ is odd.
    
    By semicontinuity, it is enough to construct a Gorenstein algebra $A$ with Hilbert function $(1,3,h_2,\dots,h_{e})$ and non-Lefschetz locus of conics with the expected codimension, and we can proceed exactly as in the proof of Theorem \ref{main CI}.
    
	Let $R/J$ be a Gorenstein algebra with Hilbert function $(1,3,h_2,\dots,h_{e})$ that comes from points, so that it is a quotient of $R/I_Z$, where $I_Z$ is an ideal associated to a set of points  $Z$ with Hilbert function $$(1,3,h_2,\dots,h_{\frac{e-1}{2}},h_{\frac{e-1}{2}}\dots ).$$ Since the first difference of the Hilbert function $(1,2,g_2,\dots,g_{\lfloor \frac{e}{2}\rfloor})$ is of decreasing type, we can assume $Z$ has the UPP by \cite{MR}.
    Then a conic $C$ is a Lefschetz conic for $R/I$ if and only if  $$\times C: [R/I_Z]_{\frac{e-1}{2}-1}\to [R/I_Z]_{\frac{e+1}{2}}$$ is injective.
    With the same notation as in the proof of Theorem \ref{main CI}, let $I_Y=(I_Z:C)$ be the ideal associated to the set of points $Y$ of $Z$ that $C$ does not pass through. The map $\times C$ is injective if and only if $[I_Y/I_Z]_{\frac{e-1}{2}-1}=0$. Since $Z$ has the UPP, this happens if and only if $C$ passes through at least
 $$h_{\frac{e-1}{2}}-h_{\frac{e-1}{2}-1}+1=h_{\frac{e-1}{2}}-h_{\frac{e+1}{2}-1}+1$$
 points of $Z$. Since $C$ is a conic, it cannot vanish at more than $5$ points of $Z$.
 So we obtain that 
 $$\cod \C_{R/I} = \min \{ h_{ \frac{e+1}{2}}-h_{\frac{e-1}{2}-1}+1, 6\}.$$
\end{proof}

Without the condition on the g-vector, the expected codimension is not necessarily achieved. In fact, if the g-vector of $A$ is not of decreasing type, then $\cod \C_A=1$. The proof proceeds in the same way that has been done for lines in \cite{main}, and we include it for completeness.

\begin{prop}
Let $A$ be a general Gorenstein algebra $A$ with Hilbert function $(1,3,\dots, h_e)$ such that the g-vector of $A$ is not of decreasing type. Then the
non-Lefschetz locus of conics has codimension $1$ in $\PP^5$.
\end{prop}
\begin{proof}
     Since the g-vector of $A$, $(1,2,g_2,\dots,g_{\lfloor \frac{e}{2}\rfloor})$,  is not of decreasing type, we can find $i<\lfloor \frac{e}{2}\rfloor$
    such that $g_{i-2}>g_{i-1}=g_i$. By \cite[Theorem 3.1]{RZ}, the generators of $I$ of degree $\leq i$ have a common factor $f$ and $\deg f=g_i$. Moreover, $g_{i-2}>g_{i-1}$ and $A$ is general, so the generators of degrees less than $i$ of $(I\colon f)$ span the ideal of a set of points $Z$ in $\PP^2$, that we can assume to be reduced \cite{RZ}.
    
    Let $C$ be a form of degree $2$ and consider the multiplication map $\times C: [A]_{i-2}\to [A]_i$. Since $i<\frac{e-1}{2}$, this map has maximal rank if and only if it is injective. Note that 
    $$[A]_j=[R/I]_j=[R/(I\colon f)\cdot f]_j=[R/(I_Z)\cdot f]_j$$ for any $j\leq i$. So we can consider the map $\times C: [R/(I_Z)\cdot f]_{i-2}\to [R/(I_Z)\cdot f]_i$. Let $Y$ be the set of points defined by the ideal $I_Z\colon C$. We have the sequence
    $$0 \to \left[\frac{I_Y\cdot f}{I_Z\cdot f}\right]_{i-2}\to \left[\frac{R}{I_Z\cdot f}\right]_{i-2}\stackrel{\times C}\to \left[\frac{R}{I_Z\cdot f}\right]_i$$
    and the map $\times C$ is injective if and only if $[\frac{I_Y\cdot f}{I_Z\cdot f}]_{i-2}=0$. Since the Hilbert function of $R/I_Z$ reaches its multiplicity at $i-g_i-2$ by \cite{Davis}, it is enough that $C$ passes through one of the points of $Z$ to get that $[\frac{I_Y\cdot f}{I_Z\cdot f}]_{i-2}\neq0$. So $\cod C_{A,i-2}=1$, and therefore $\cod C_A=1$.
\end{proof}
Without the hypothesis of generality, we do not know if there exists a Lefschetz conic, so the codimension of $\C_A$ could be zero. In fact, for Gorenstein algebras $\kk[x_1,x_2,x_3]/I$, even the WLP is an open question.

\bibliographystyle{amsalpha}
\bibliography{bibfile}

\end{document}